%% file: main.tex
\documentclass[12pt,a4paper]{article}

\usepackage{hyperref}
\usepackage{amsmath,amssymb,array}
\usepackage{amsthm}
\usepackage{cleveref}
\newtheorem{proposition}{Proposition}
\newtheorem{remark}{Remark}

\usepackage{booktabs}
\usepackage{graphicx}
\usepackage[width=1\textwidth]{caption}
\usepackage[a4paper, left=1.6cm, right=1.6cm, top=2.6cm, bottom=2.6cm]{geometry}
\usepackage{float}

\usepackage[square,numbers]{natbib}
\newcommand\labmarg[1]{
\label{#1}\mbox{} 
}

\usepackage{enumitem}

\usepackage{authblk}
\title{{\tt Splinets 1.5.0} -- Periodic Splinets}
\author[1]{Hiba Nassar}
\author[2]{Krzysztof Podg\'orski}
\affil[1]{Department of Applied Mathematics and Computer Science, Technical University of Denmark, Denmark}
\affil[2]{Department of Statistics, Lund University, Sweden}
\date{}

\begin{document}
\input{ARXIVPeriodic_Splinets}
\end{document}

%% file: ARXIVPeriodic_Splinets.tex
\maketitle

\abstract{Periodic splines are a special kind of splines that are defined over a set of knots over a circle and are adequate for solving interpolation problems related to closed curves. 
This paper presents a method of implementing the objects representing such splines and describes how an efficient orthogonal basis can be obtained. The proposed orthonormalized basis is called a periodic splinet in the reference to earlier work where analogous concepts and tools have been introduced for splines on an interval. 
Based on this methodology, the periodic splines and splinets are added to the earlier version of the R package {\tt Splinets}. Moreover, the developed computational tools have been applied to functionally analyze a standard example of functional circular data of wind directions and speeds.}

\section{Introduction}
\labmarg{sec:intro}
Spline functions, or splines, for shortness, are piecewise polynomials that are smooth up to the largest order of the polynomial less one, \cite{Deboor}, and constitute a good tool for modeling  contours and surfaces in many areas of application such as signal and image processing, computer vision and computer graphics \cite{dierckx, biswas}.
The $B$-splines constitute a natural functional base for the splines of a given order and have been widely used due to their flexibility and efficiency.  They  have interesting properties such as  positivity, compact and local support, differentiable up to a certain level (depending on the spline order), and convenient computational relations. 
The $B$-splines are constructed as polynomial pieces having the same degree and connected smoothly at  points $\xi_{0}<\xi_{1}<\dots <\xi_{n}< \xi_{n+1}$, referred to as knots.  They can be effectively evaluated in a recursive way for any degree  by means of the Cox-de Boor  formula, which is presented here under a convenient assumption that the splines of the given order have the derivative up to this order (exclusive) equal to zero at the endpoints $\xi_0$ and $\xi_{n+1}$.
Namely, given a knot sequence $\xi_{0}<\xi_{1}<\dots <\xi_{n}< \xi_{n+1}$, the B-splines of order $0$ are defined by indicator functions 
\begin{equation}
\label{eq:indi}
B_{\ell,0}=\mathbb I_{(\xi_{\ell},\xi_{\ell+1}]}, ~~~\ell=0,\dots n.
\end{equation}
The following recursion relation leads to the definition of the splines of arbitrary order $k \le n$
\begin{equation}
\label{eq:recspline}
B_{\ell,k}(x)
=
 \frac{ x- {\xi_{\ell}}
  }{
  {\xi_{\ell+k}}-{\xi_{\ell}} 
  } 
 B_{\ell,k-1}(x)+ \frac{{\xi_{\ell+1+k}}-x}{
 {\xi_{\ell+1+k}}-{\xi_{\ell+1}} 
 } 
 B_{\ell+1,k-1}(x).
\end{equation}
In this convention at the boundaries, we note that the $k$-order $B$-splines span an $n-k+1$-dimensional functional space of the splines with the same boundary conditions at $
\xi_0$ and $\xi_{n+1}$.

The splinets constitute an efficient orthonormalization of the $B$-splines, preserving some of the favorable properties of the $B$-spline basis, namely  locality and computational efficiency, \cite{splinets}. 
Locality that is exhibited through small size of the total support of a splinet and computational efficiency that follows from a small number of orthogonalization procedures needed to be performed on the B-splines to achieve orthogonality.

Periodic splines, \cite{graham, spath}, refer to the splines that are defined over a set of knots  with the last knot coinciding with the first one, i.e. $\xi_{0}= \xi_{n+1}$ as illustrated in Figure~\ref{circle}. 
\begin{figure}[t!]
  \centering
\includegraphics[width=0.45\textwidth]{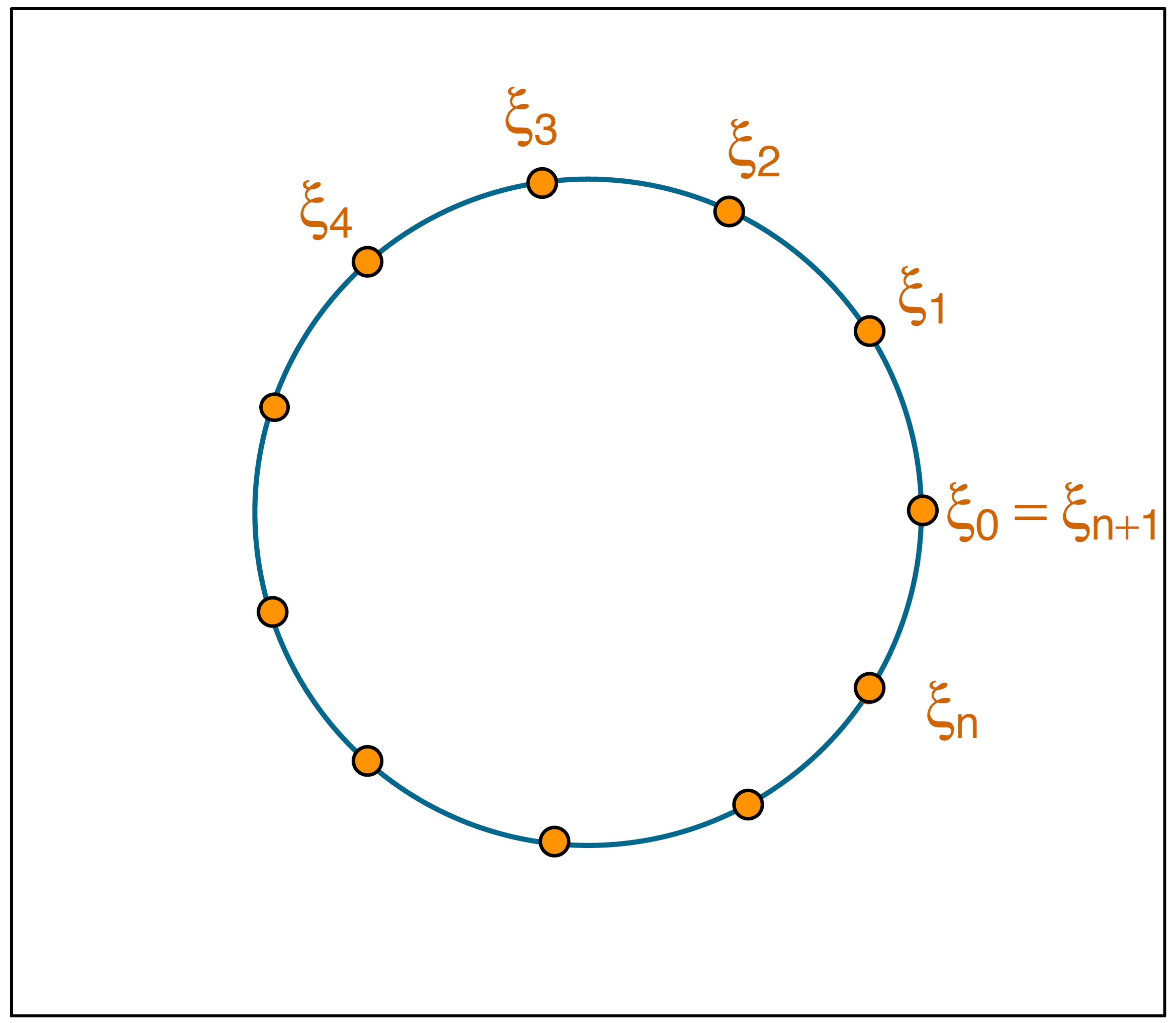}
  \caption{\small Knots wrapping for constructing closed loop.   }
  \label{circle}
  \end{figure}
A {\em regular} spline is built on a set of knots that spread over an open interval (has two endpoints), while a periodic spline is built on a set of knots that forms a closed loop. 
There are many advantages of using a closed loop. 
By the way the periodic splines are constructed, they are periodic and  preserve curvature continuity along the entire curve. 
Moreover, periodic splines offer a good visualization using polar coordinates: the radius $r$ and the azimuth $\theta$ with respect to some pole. 
A common application of them would be data that are naturally driven by the polar coordinates but they can be also a non-frequency based alternative to trigonometric bases in modeling stationary data.

The organization of the material in this note is as follows. We start with a brief account of splinet bases and its construction. 
Section~\ref{sec:periodic} presents the proposed periodic splines and splinets. Section~\ref{graphical} provides an explanation of the graphical implementation in the R-package {\tt splinets}. Then Section~\ref{eff} presents the efficiency of the periodic splinets in representing circular data.  The paper ends with Section \ref{wind}, where an application of the periodic splines to represent real wind data is presented.

\section{Basics on the splinets}
\labmarg{sec:basic}
One deficiency of the $B$-splines is that they are not orthogonal. 
The lack of orthogonality makes it harder to represent the data in the $B$-spline basis. 
It requires solving a linear equation for the coefficients of the representation. Although algorithms have been developed for this purpose, it is more efficient to first orthogonalize $B$-spline and then find the representation of a function using the orthogonal basis.
The efficiency comes from the fact that for an orthogonal basis, it is enough to find the inner products between the data and the elements of the basis and no additional equation needs to be solved. 
In the literature, there are three different orthogonalization methods for the $B$-spline. 
We mention the first two briefly here, for more details see \cite{mason1993orthogonal}.
while for the third method that will be used in this paper, the general idea is explained while further details and results can be found in \cite{splinets}   :
\begin{itemize}
\item One-sided orthogonalization.  This orthogonalization can be the simplest described as the Gram-Schmidt (GS) orthogonalization, when applied to the special base formed from the $B$-splines.
Some computational gain can be obtained due to the locality (and thus partial orthogonality) of the $B$-splines.
\item Two-sided orthogonalization -a symmetrized GS method.  This orthogonalization can be represented by applying the one-sided orthogonalization from the two ends of the interval and then properly modified for a few central $B$-splines leads to a two-sided method.
Due to the locality of the $B$-splines this further improves the efficiency. 
\item The splinets is defined as the orthogonal basis obtained by a structured orthogonalization used in the dyadic algorithm.
It further exploits the locality of the $B$-splines that led to efficiencies in the other two orthogonalization methods.
It is performed through grouping the $B$-splines in a dyadic pyramid and then recursively applying orthogonalization starting from the bottom of the pyramid.
Consequently, this orthogonal basis of splines is better visualized as a dyadic net of orthogonalized functions rather than a sequence of them and hence the name `splinet'. 
For a detailed explanation of the dyadic algorithm and the splinets we refer to \cite{splinets}. 
A visualization of the pyramid like structure can be seen in Figure~\ref{periodic_splinet}, where the same algorithm is applied to the periodic $B$-splines.  
The splinet method is preferred over the previous two methods since it is computationally more efficient and it preserves the locality properties similar to those featured by the $B$-splines.
\end{itemize}

\section{Periodic splines and splinets}
\labmarg{sec:periodic}
In this section, we will explain how to construct periodic $B$-splines and periodic splinets using regular $B$-splines and splinets, respectively. 
A general periodic spline is then obtained as a linear combination of these two functional bases, i.e. the periodic splines constitute a linear space that is spanned by the periodic $B$-splines or, equivalently, by the periodic splinet.  
\subsection{Periodic $B$-splines}
Assume the following set of knots $\xi_{0}<\xi_{1}<\dots <\xi_{n}< \xi_{n+1}$. The set of $B$-splines with these fixed knots consists of  $n + 1- k$ regular splines of order $k$. 
In periodic splines, the set of the knots spread over a closed circle with $\xi_0 = \xi_{n+1}$ as in Figure~\ref{circle}. Moreover, the length of the arc between any two knots preserves the normalized distance of the corresponding knots on the interval.
The method of constructing periodic splines using regular $B$-splines is based on an extension of knots. The set of knots needs to be extended by $k$ knots, $\xi'_{1}, \xi'_{2}, ...\xi'_{k}$, located before the initial endpoint that preserves the distances between the last $k$ knots.  Similarly, a number of $k$ knots, $\xi'_{k+1}, \xi'_{k+2}, ...\xi'_{2k}$, are added after the terminal endpoint that preserves the distances between the first $k$ knots. Those knots are called extra knots. Hence, the extended set of the knots is  
\begin{equation}
\label{eq:extension}
\xi'_{1}<\xi'_{2}< \dots < \xi'_{k}<\xi_{0}<\xi_{1}<\dots <\xi_{n}< \xi_{n+1}<\xi'_{k+1}< \xi'_{k+2}< \dots <\xi'_{2k},
\end{equation}
where the distance between the extra knots satisfies the following conditions
\begin{equation}
\label{eq:ext2}
    \begin{aligned}
     &\xi'_{k+1}  - \xi_{n+1}=\xi_{1}-\xi_{0},   &&\xi_{0}-\xi'_{k}=\xi_{n+1}-\xi_{n},\\
     &\xi'_{k+i}-\xi'_{k+i-1}= \xi_{i}- \xi_{i-1}, &&\xi'_{i+1}-\xi'_{i}= \xi_{n-k+i}- \xi_{n-k+i-1}.
    \end{aligned}
\end{equation}

Due to the previous restriction the first $2k$ intervals and the last $2k$ intervals generated in between extended knots have the same length of segments respectively, i.e. the first segment from the first $2k$ intervals equal to the first segment in the last $2k$ intervals, and the second equal to the second, etc. Hence, the set of $k$ $B$-splines ($k$-tuplet) defined over the first $2k$ intervals are identical to the last $k$-tuplet of $B$-splines defined over the last $2k$ intervals  (see Figure~\ref{extra_knots}). 
This duplication is interpreted as the periodicity (the $B$-splines can be now extended periodically over the entire line by repeating the ordered sequence of splines).
Alternatively, if one wraps the interval along a circle then the first $k$-tuple will overlap the last one and thus they can be identified as a one $k$-tuple.

\begin{figure}[t!]
  \centering
\includegraphics[width=0.9\textwidth,height=0.5\textwidth]{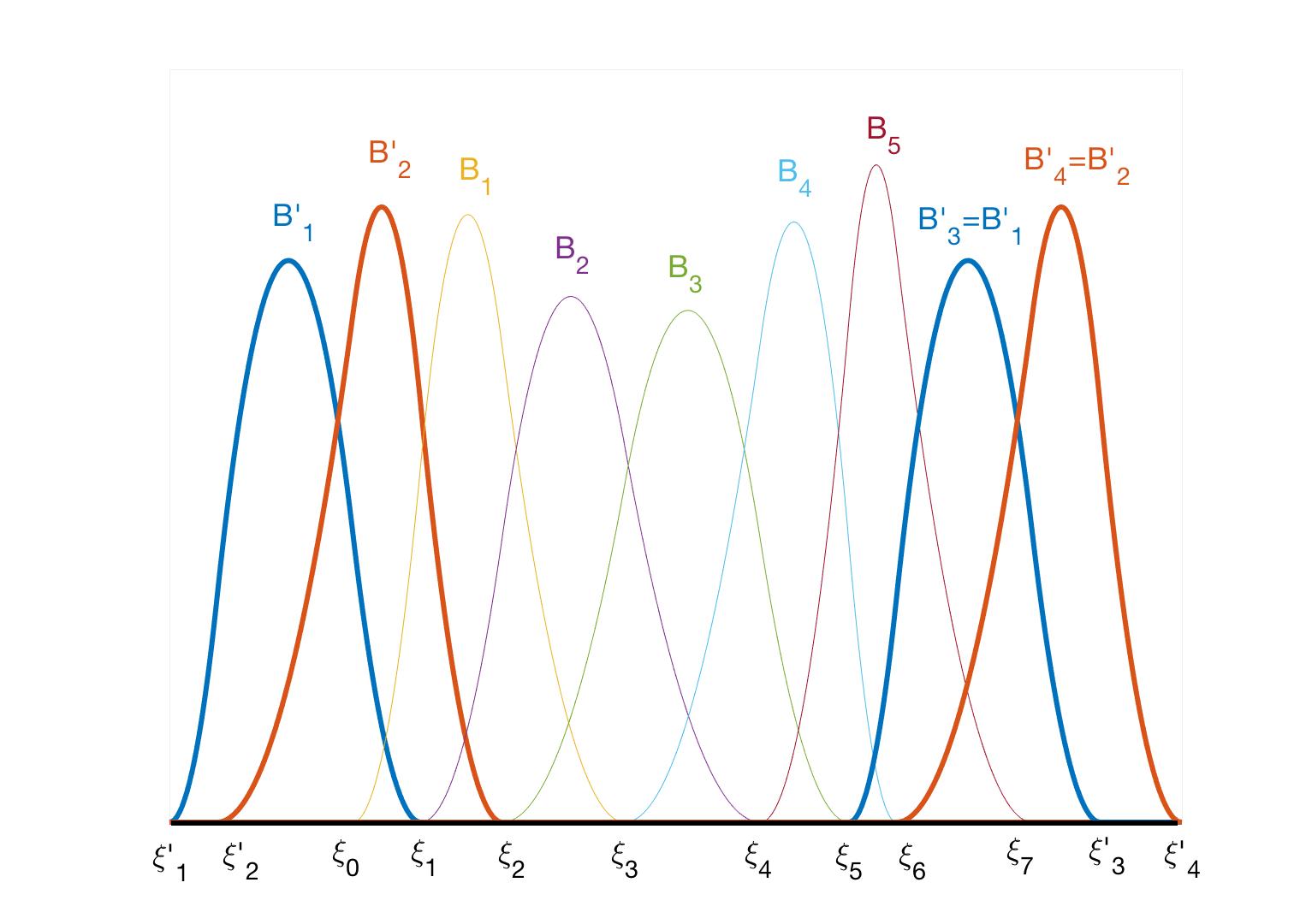}
  \caption{\small Adding extra knots at endpoints preserving distance, and the resulting pairs identical $B$-splines of order 2 at the beginning and at the end. These two pairs illustrate the periodicity of the $B$-splines and on the circle they coincide.}
  \label{extra_knots}
  \end{figure}

The periodic $B$-splines consist of the regular splines that are built over the set of knots $\xi_{0}, \dots,  \xi_{n+1}$, and an extra $k$-tuplet of $B$-splines, $B'_1, B'_2, \dots, B'_k$, to complete the set of $B$-splines. 
To summarize
\begin{itemize}
\item The $n+2$ knots are extended to $n+2k+2$ knots through \eqref{eq:extension} and \eqref{eq:ext2}.
\item Over this extended set of knots the regular $B$-splines are built. There will be $n+k+1$ of them.
\item The first and the last $k$-tuple have identical shape and are considered to be the same on the circle. They represent the periodicity of the $B$-splines since other splines can be repeated periodically on both sides of the real line. Each spline replicated in this way is considered as one periodic spline with the period $T=\xi_{n+1}-\xi_0$. 
\end{itemize}
The first and the last $k$-tuple when restricted to the original domain $[\xi_0,\xi_{n+1}]$ represent parts of a single $k$-tuple, the first or the last in the extension, cut at the $\xi_{0}=\xi_{n+1}$ on the circle.
 For example, the support of the first extra $B$-spline spreads over the $k+1$ segments  at the beginning and at the end of the interval $[\xi_0, \xi_{n+1}]$.  Consequently,  this spline has the knot $\xi_0$, or equivalently $\xi_{n+1}$, in its support.

We note the dimension of the periodic splines with the imposed periodic boundary conditions is $n + 1$ and the counts is made as follows: there are $n+ 1-k$ regular $B$-splines with the zero-boundary condition on the interval $[\xi_0, \xi_{n+1}]$ and $k$ extra splines to complete the loop.

\begin{figure}[t!]
  \centering
\includegraphics[width=0.5\textwidth,height=0.5\textwidth]{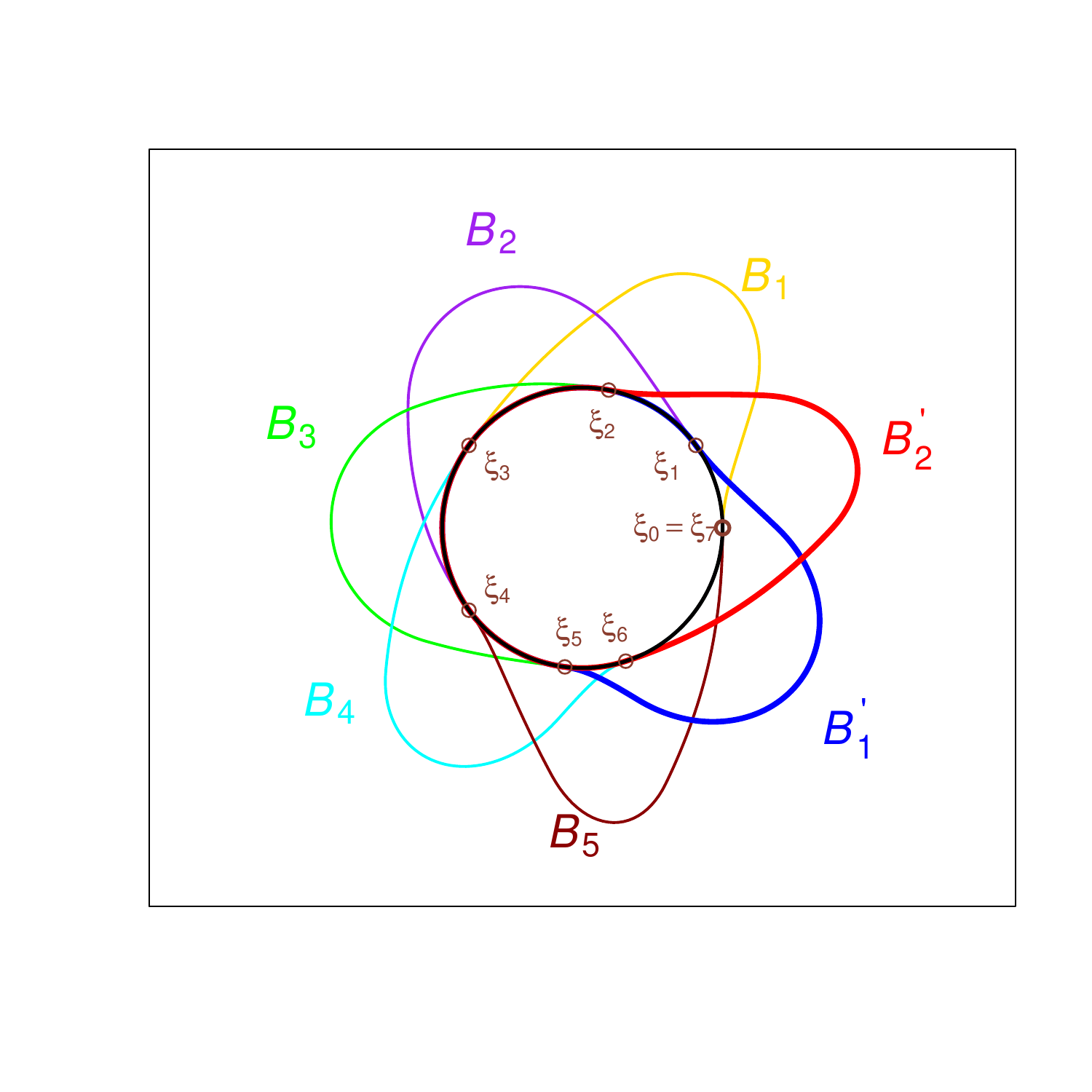}
  \caption{\small The periodic splines corresponding to those in Figure~\ref{extra_knots} in the polar coordinate graph.}
  \label{per}
  \end{figure}
Figure~\ref{per} illustrates the periodic splines that plot  in the polar coordinates, see also Section~\ref{graphical}. 
The two splines, in bold line, $B'_1,  B'_2 $ are the extra splines to complete the loop and overlapping with the pair $B_{3}',B_4'$ that is not explicitly represented in the graph.


\begin{figure}[htbp]
  \centering
\includegraphics[width=0.45\textwidth,height=0.5\textwidth]{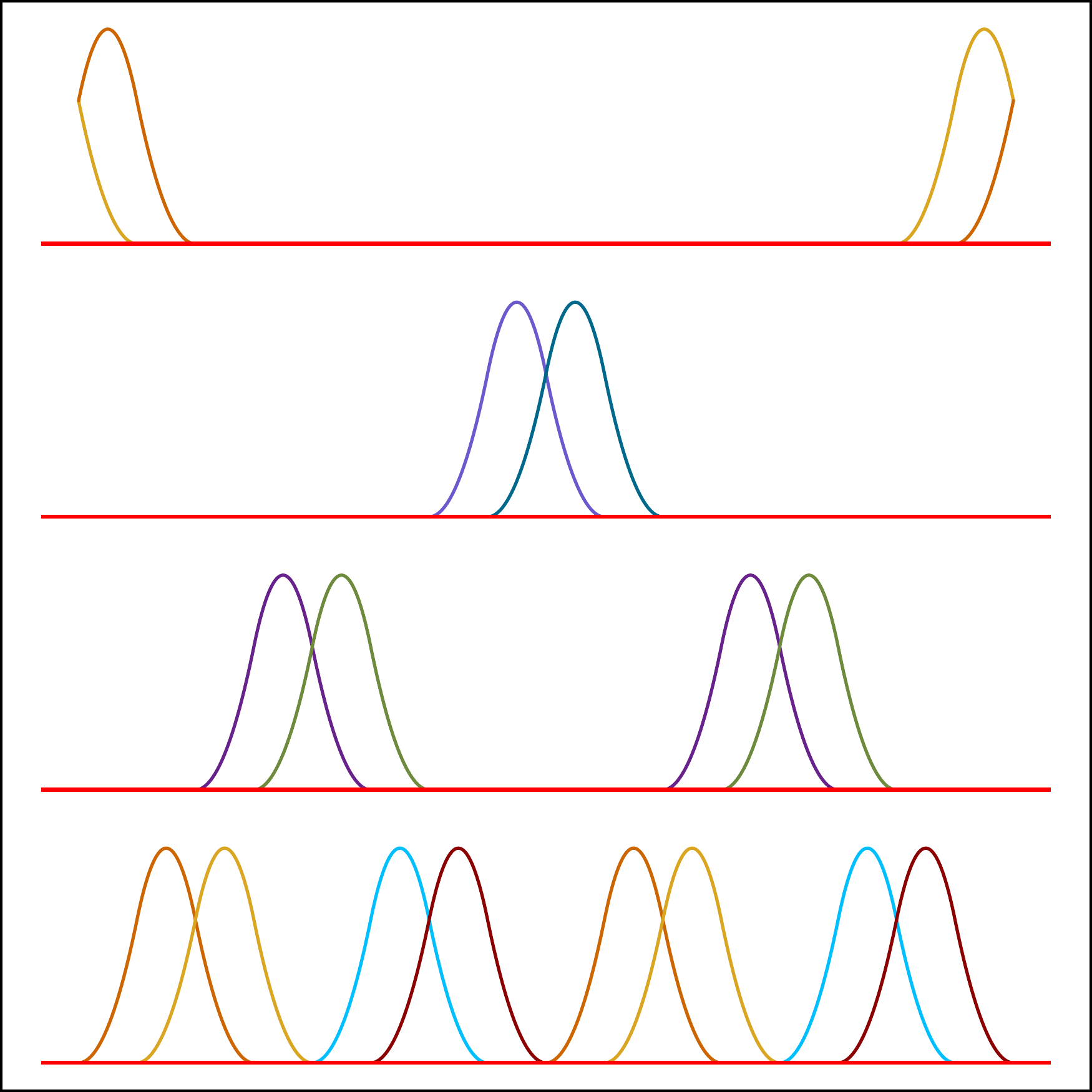}
\includegraphics[width=0.45\textwidth,height=0.5\textwidth]{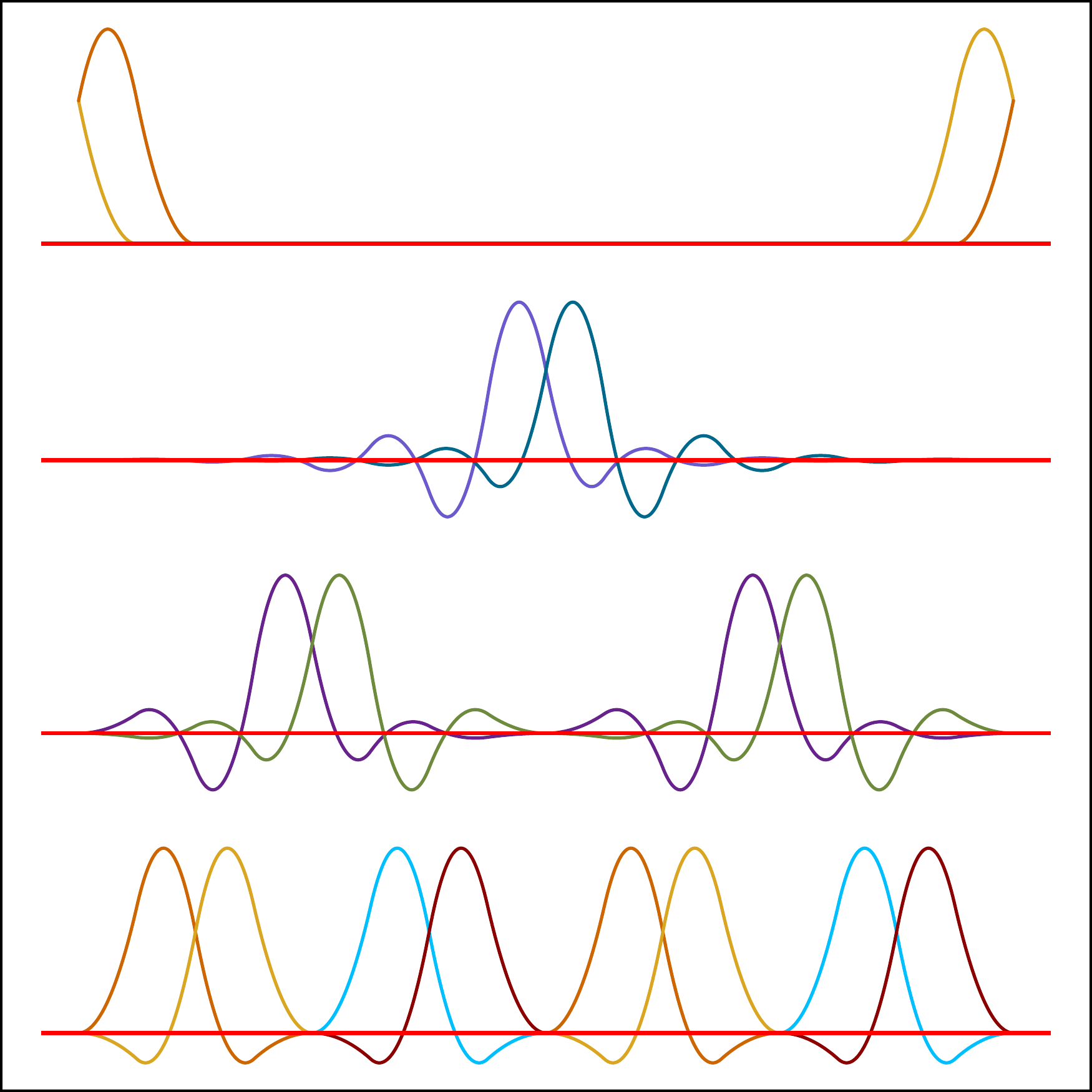}\vspace{6mm}\\
\includegraphics[width=0.45\textwidth,height=0.5\textwidth]{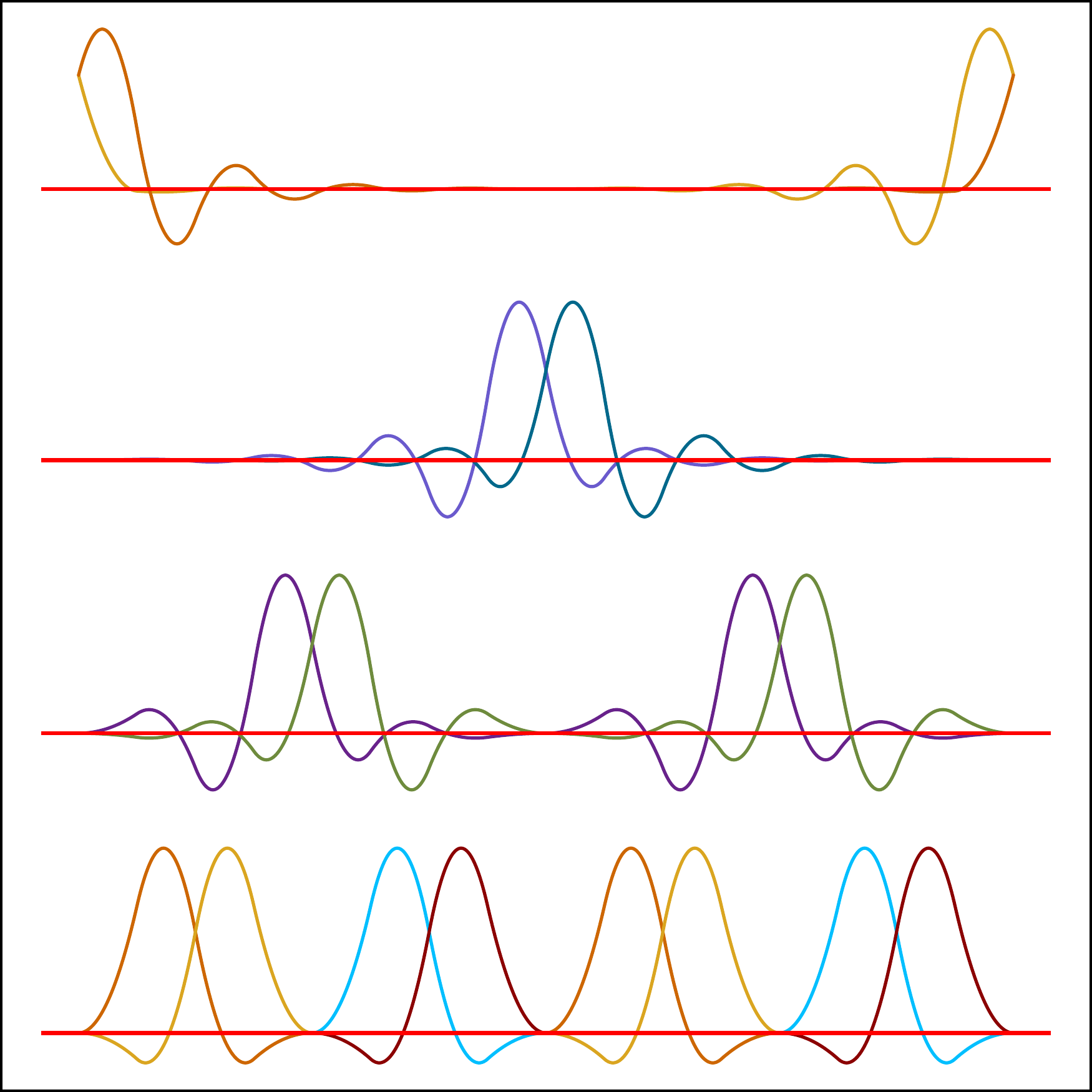}
\includegraphics[width=0.45\textwidth,height=0.5\textwidth]{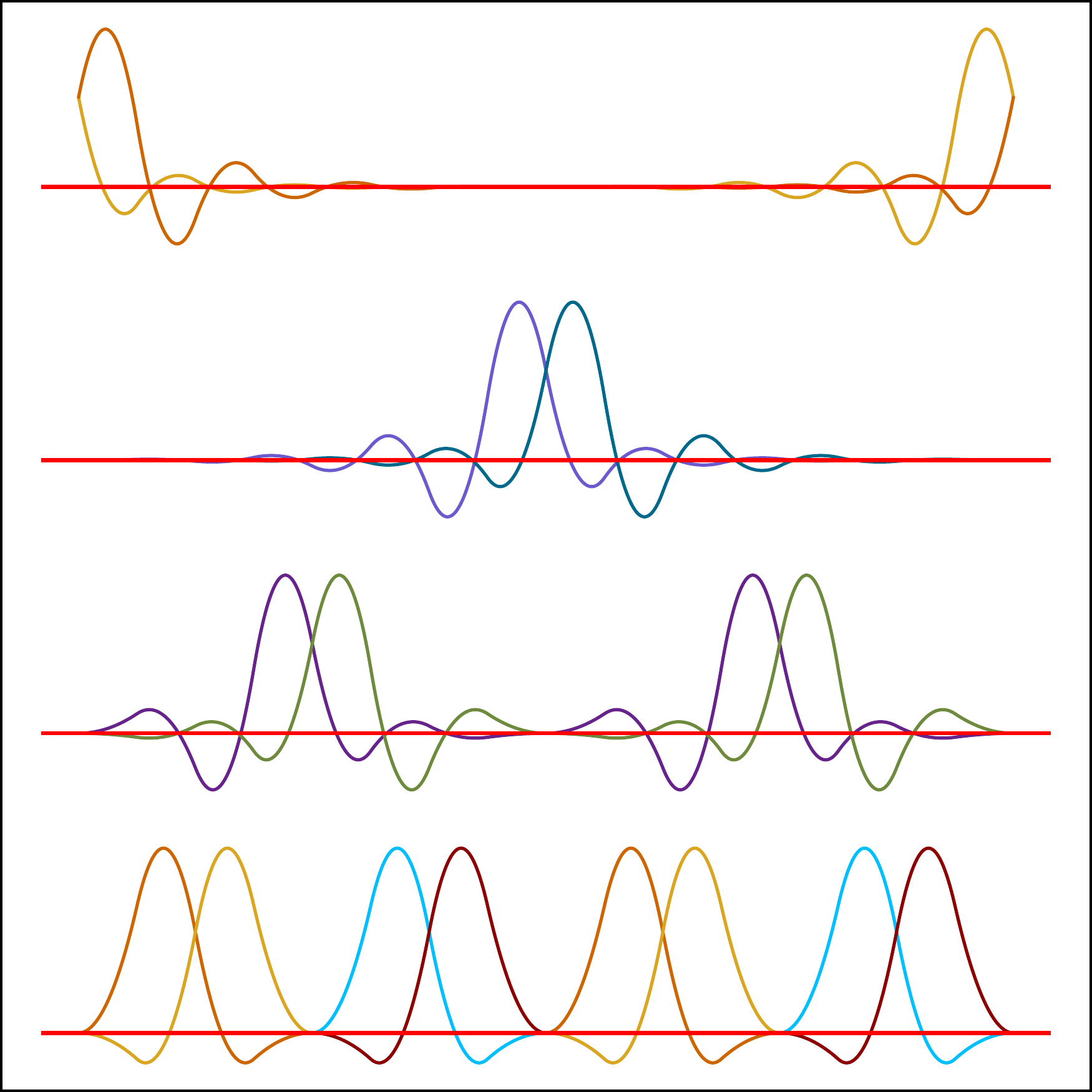}
  \caption{\small Steps in the construction of a periodic splinet.
  {\it (Top-left):} The dyadic pyramid before orthogonalization with regular $B$-splines in the first three rows and the extended $B$-splines in the top row; 
  {\it (Top-right):} The regular splinet in the three bottom rows and unorthogonalized extended $B$-splines in the top row;
  {\it (Bottom-left):} The regular splinet (bottom three rows) with the extended splines orthogonalized with respect to it in the top row;
  {\it (Bottom-right):} The final periodic splinet with the extended splines mutually orthogonalized in the top row.}
  \label{per_net}
  \end{figure}

\subsection{Periodic splinet}
 As explained above the periodic $B$-splines consist of two sets of splines:  $n-k+1 $ regular splines and $k$ extra splines which are conveniently duplicated on the extended knot representation.
 Instead of duplication and extending the domain, one can simply return to the original domain $[\xi_0,\xi_{n+1}]$ and consider the pairs of extended $B$-splines as the same spline that does not have the zero boundary condition. 
 These splines have the periodic boundary conditions, i.e. the derivatives  up to the $k$th order are the same at the endpoints.
 To keep the `net' structure,  an extra level is added on top of all the levels of the regular splinet. On the extra level, there will be the $k$-tuplet of the orthogonalized extra splines.
 Figure~\ref{per_net}~{\it (top-left)} presents the regular $B$-splines in the first three rows and the extra $k$-tuplet in the top row ($k=2$ in this example), in which we see the two splines with periodic boundary conditions (orange and yellow lines). 
The method of obtaining periodic splinets from periodic splines can be summarized in the following steps.
\begin{description}
\item[1)]  We first build the $n-k+1 $ orthonormal splines ($O$-splines)  from the regular $B$-splines, as explained in \cite{splinets} and implemented in the {\tt Splinets}-package. This group of $O$-splines is called the regular splinet, see Figure~\ref{per_net}~{\it (top-right)}.
\item[2)] To `close the loop'   the extra $k$-tuplet of splines at the top row is orthogonalized with respect to the regular splinet. Due the locality and small supports that regular splinets have, see Chapter 5 \cite{splinets}, the extra $k$ splines have common supports only with the first and the last $k$-tuples on each level. Hence, this step is achieved by orthogonalizing the extra $k$-tuplet of $B$-splines with respect to $2N$ $k$-tuplets from the regular splinet, where $N$ is the number of levels. The outcome of this step is presented in Figure~\ref{per_net}~{\it (bottom-left)}.
\item[3)] The extra two $k$-tuples of splines, while orthogonal to everything else, are not orthogonal to each other. 
This step of the construction of the splinet concludes by orthogonalizing the extra $k$ splines at the top with respect to each other. 
For this, we use the symmetric Gram-Schmidt orthogonalization, which is directly implemented in the {\tt Splinets}-package. We note that the periodic condition for the extended splines (all derivatives up to the $k$th order are the same at the endpoints) is preserved due to the zero boundary (and thus also periodic) conditions of the regular splinet. 
These orthogonalized extra splines will have full support and thus are on the maximal range together with the $k$-tuplet on the second last level.
The final outcome, the periodic splinet, is presented in Figure~\ref{per_net}~{\it (bottom-right)}
\end{description}

\subsection{Periodic splines}
By setting the set of knots and using the extension of these knots as described in the previous section, not only periodic bases such as the $B$-splines and the splinet, but also any periodic spline can be expressed as a regular spline over the extended knots. 
The method of representing the regular splines through efficient Taylor expansions around the knots has been discussed in detail in \cite{Podgorski} and was implemented in the original version of {\tt Splinets 1.0.0}.

Alternatively, once the periodic bases have been determined they can represent any periodic splines by the linear combination of the basis elements, which is another efficient representation of the periodic splines as vectors in $n+1$ dimensional Euclidean space. 

\section{Graphical visualization}
\labmarg{graphical}

  \begin{figure}[t!]
  \centering
\includegraphics[width=0.4\textwidth]{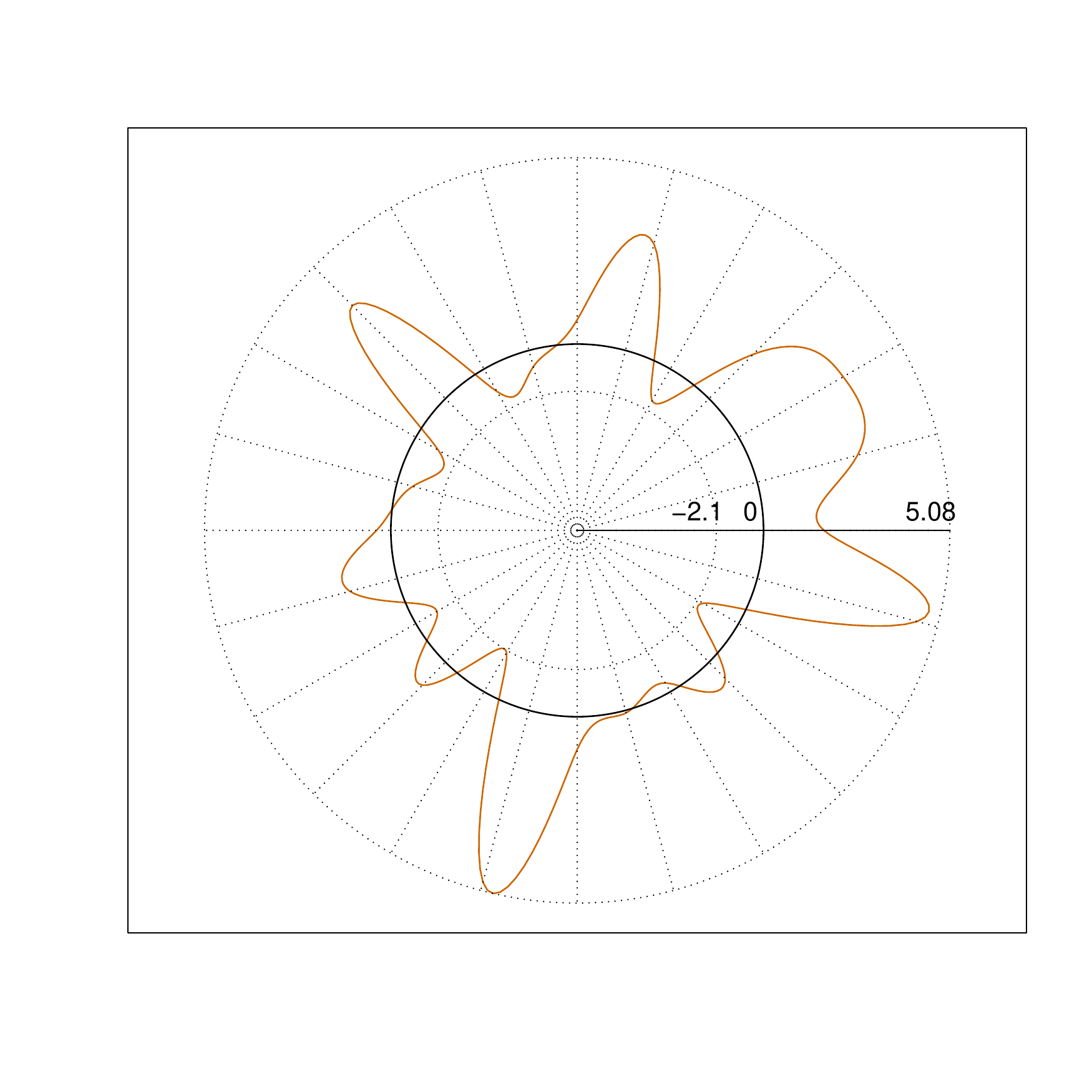}
\includegraphics[width=0.4\textwidth]{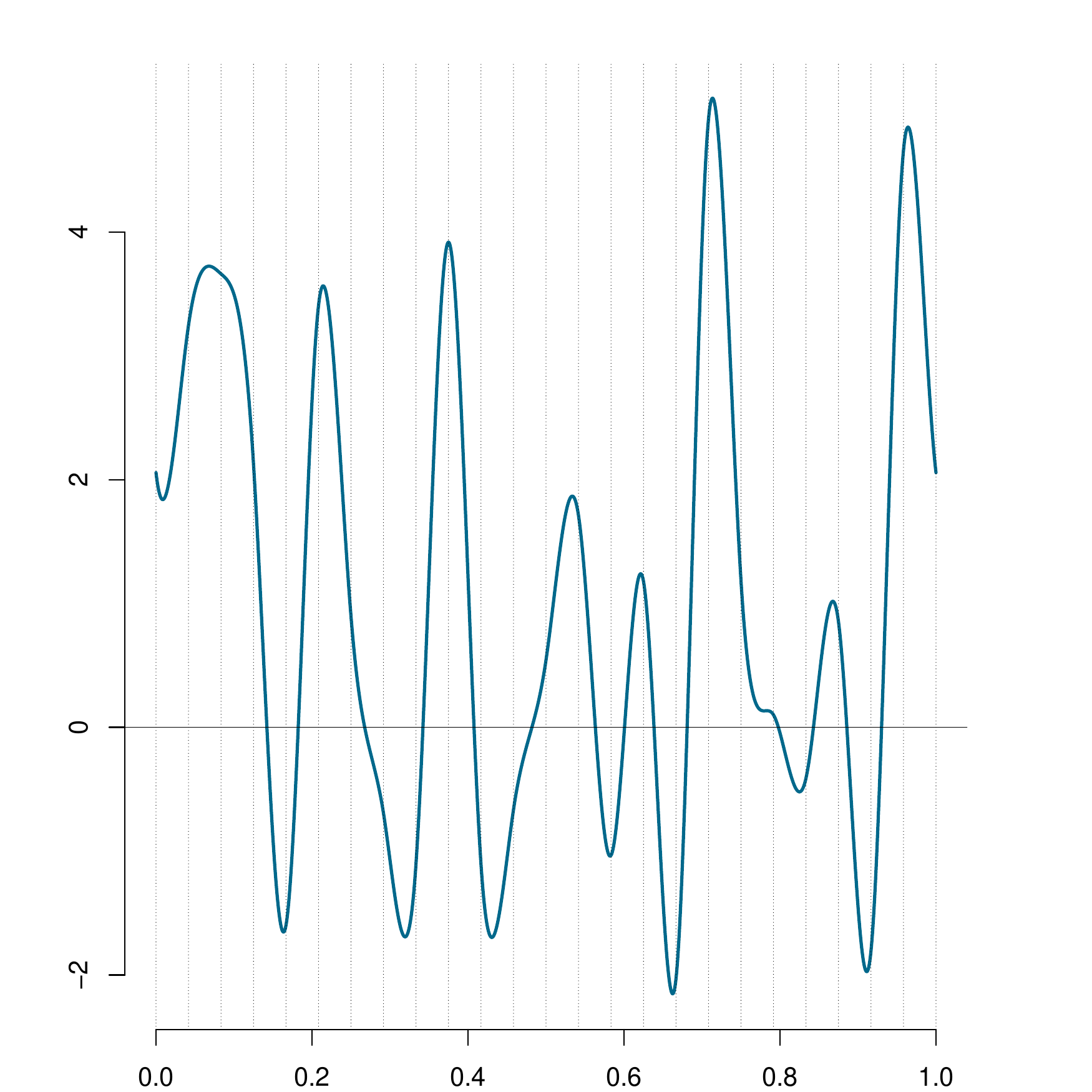}
\\
\includegraphics[width=0.4\textwidth]{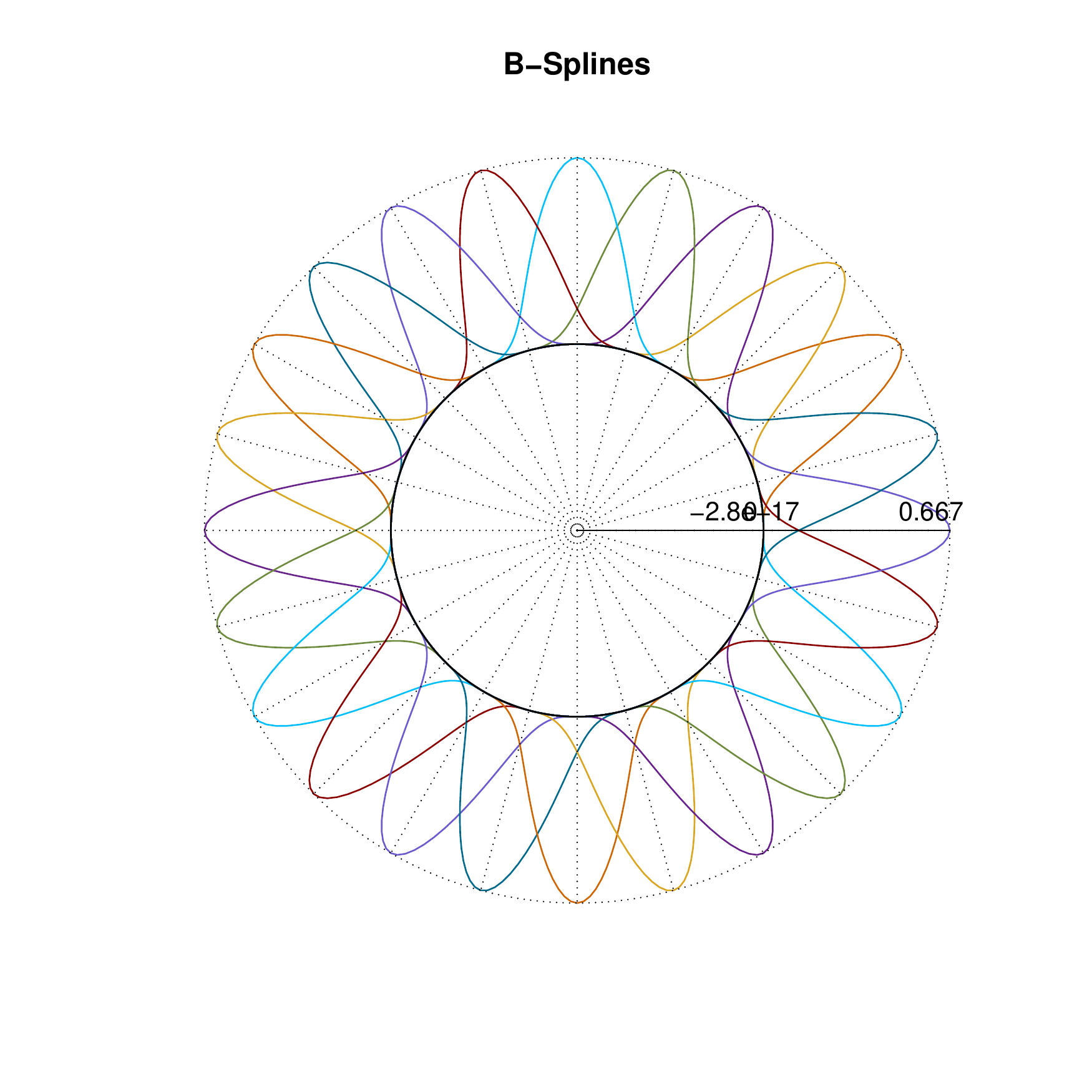}
\includegraphics[width=0.4\textwidth]{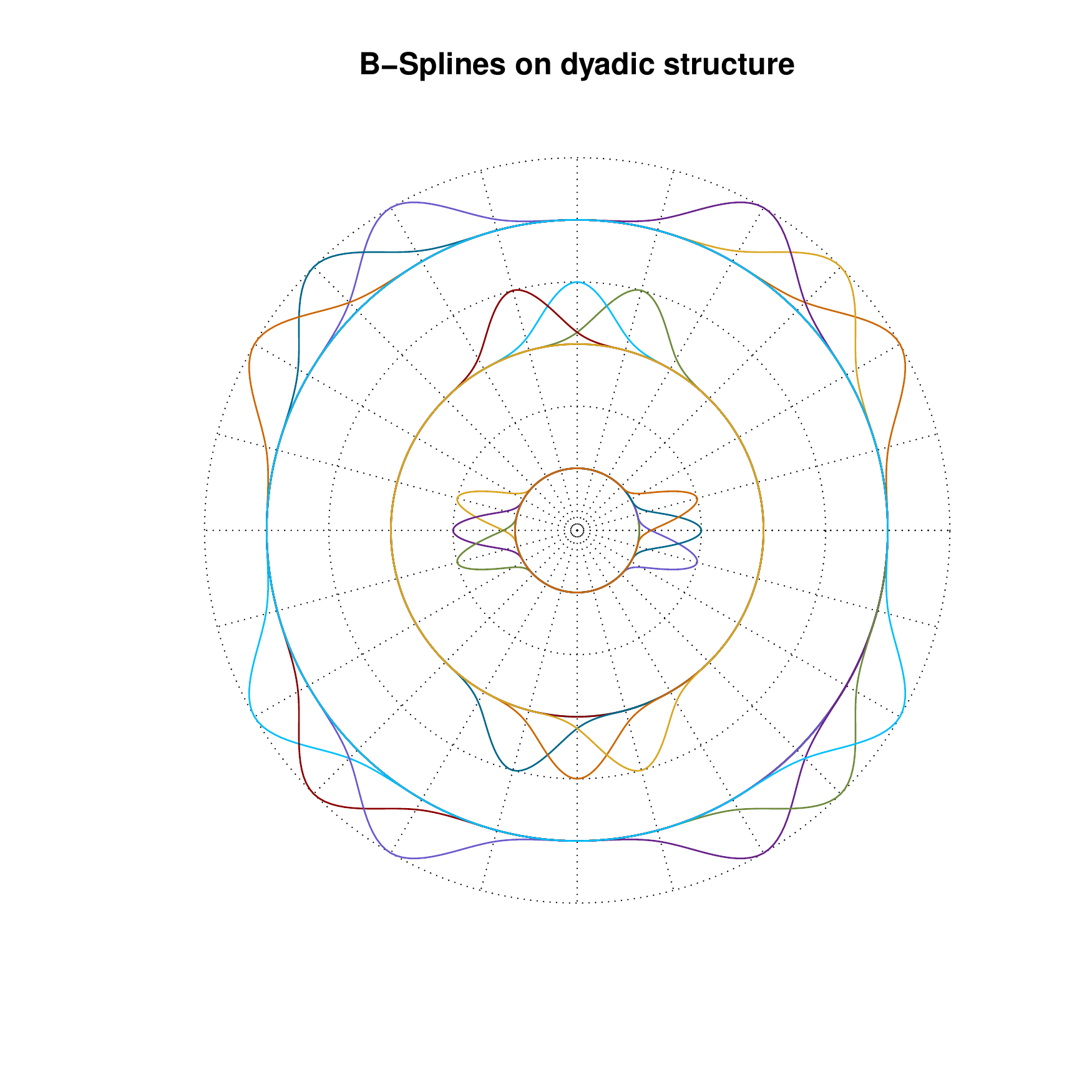}
  \caption{\small A periodic function in the polar coordinate representation {\it (top-left)} and in the cartesian coordinate {\it (top-right)}. The dotted segments in both graphs are located at the knots. In the bottom graphs, two visualizations of the $B$-splines chosen to decompose the presented function: {\it (left)} the $B$-splines on a single plot; {\it (right)} the same $B$-splines on a dyadic structure. }
  \label{fig:periodic_function}
  \end{figure}
In this section, we present a method of visualizing periodic splines, periodic splinets, and the decomposition of a periodic  function by a splinet. 
Our approach is based on the polar coordinates since they provide the  most simple and natural way to present the periodicity. 
The location of a point on the plane is determined via the coordinates $( \theta, r)\in (0,2\pi] \times [0,\infty)$, which are referred to as the azimuth ($\theta$) and the radius ($r$). 
Our goal is a graphical representation in the plane and in the polar coordinates of a collection of functions $\{B_i, i\in \mathcal I\}$, each defined on a circle. 
For this, we define a one-to-one map $T$  over a  closed rectangular region  
$$
R=\left\{ (x,y) \in \mathbb R^2 | \quad 0 < x \leq 1, \quad m\stackrel{def}{=}\min_{i, x}{B_i(x)} \leq y \leq M\stackrel{def}{=}\max_{i,x}{B_i(x)}\right\}, 
$$
transfers each point $(x,y)$ in $R$ to a corresponding  $( \theta, r)$ in the polar coordinate 
\begin{equation}
\label{transform}
(x,y) \stackrel{T}{\longrightarrow} (\theta= 2 \pi x  , r= \exp (a y)), 
\end{equation}
where $a= \log (2)/M$.

  \begin{figure}[t!]
  \centering
\includegraphics[width=0.4\textwidth]{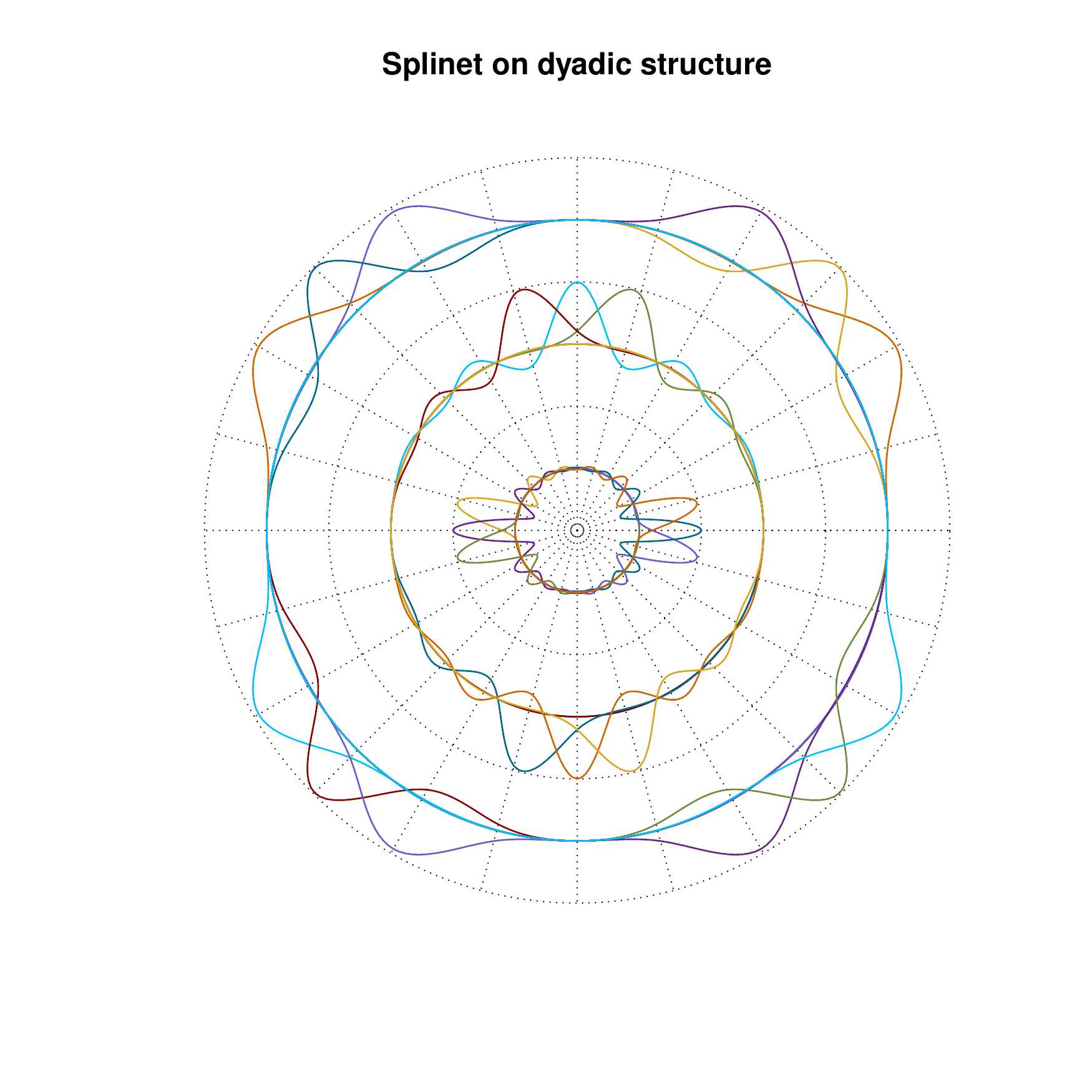}
\includegraphics[width=0.4\textwidth]{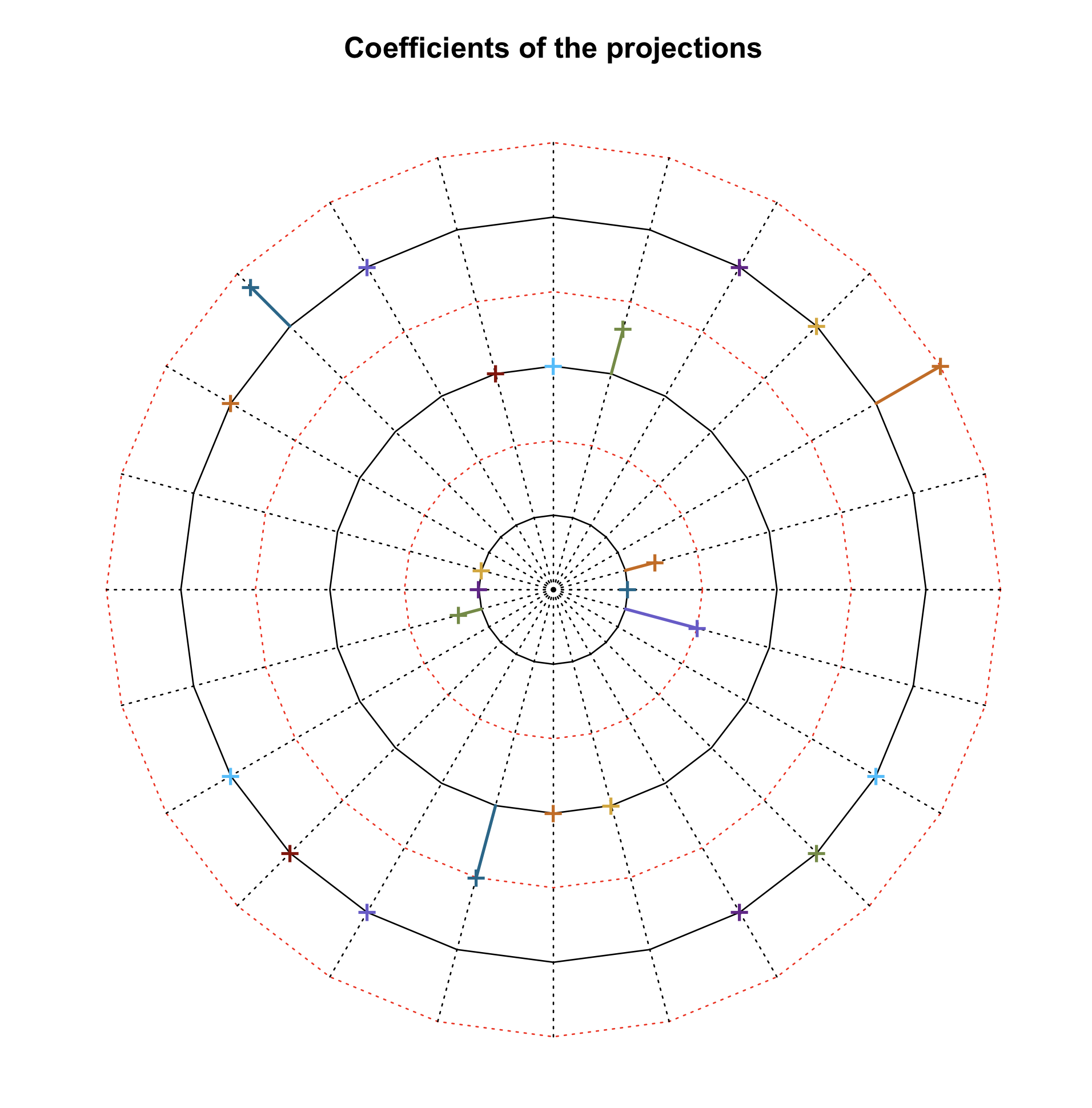}\\
\includegraphics[width=0.4\textwidth]{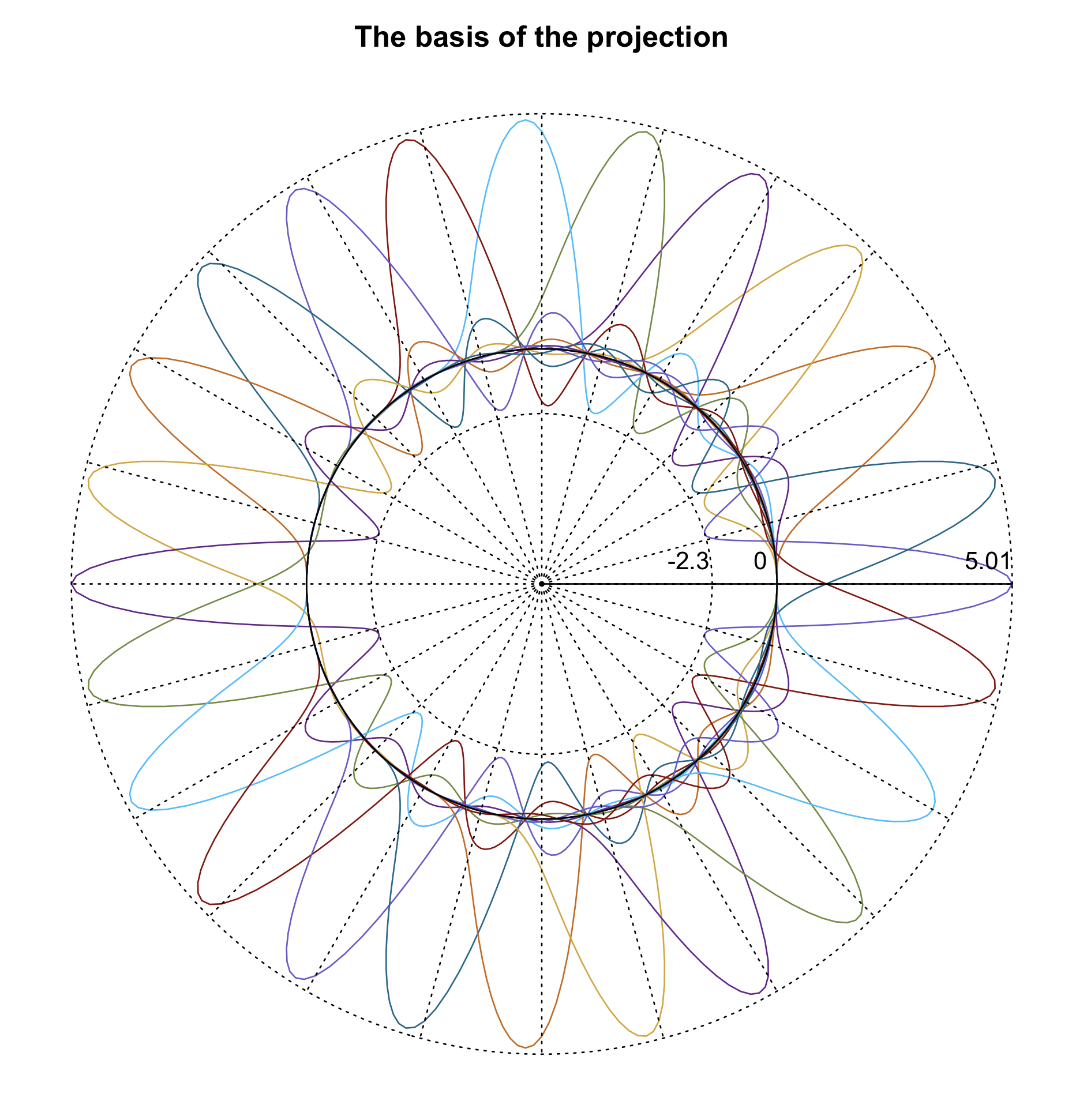}
\includegraphics[width=0.4\textwidth]{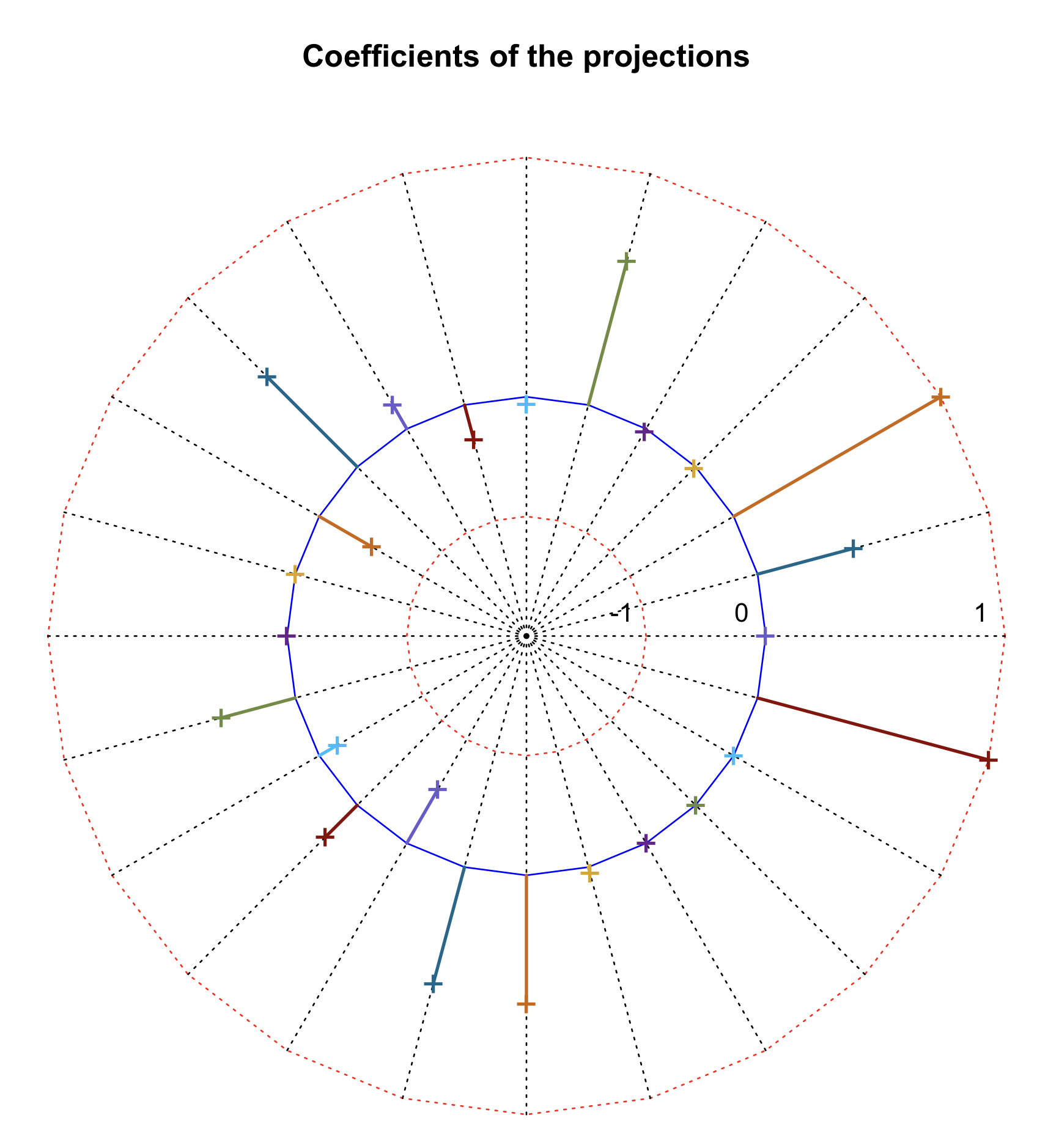}
  \caption{\small {\it (Left-Top):}  A periodic splinet  built  in the polar coordinate representation that utilizes the dyadic structure. {\it (Right-Top)}:
  The splinet decomposition of the function is presented in Figure~\ref{fig:periodic_function}. The colors of the spikes correspond to the colors of the splinet-element and crosses on the spikes mark the contributions of the splinet-elements.
  The bottom graphs are analogous to the top ones but with the dyadic base replaced by the two-sided orthogonal basis.}
  \label{periodic_splinet}
  \end{figure}

The map $T=(T_1,T_2)$ transfers the set of the knots, $\xi_{0}<\xi_{1}<\dots <\xi_{n}< \xi_{n+1}$ that are spread over an interval, to spread them over the circle $\theta_{k}=T_1(\xi_{k})$,  $k=0, 1, \dots, n$, so that the arc length between each pair of the neighboring knots is proportional to the distance between the corresponding knots on the interval. 
Moreover,  the corresponding angles for $\xi_0 $ and $\xi_{n+1}$ are   $\theta_0= 0$ and $\theta_{n+1}=2\pi$, respectively, i.e. $\xi_0 $ and $\xi_{n+1}$ coincide on the unite circle. 
One other advantage of this choice of the map $T$ is that it gives a clear visualization when the  periodic splines have positive or negative values. 
In this visualization, one can see the positive values of the periodic splines lie outside the unit circle while the negative values lie inside the unit circle ($r=\exp(a y) $ is larger than one when $y>0$ is positive and  smaller than one when $y<0$).
Moreover, the constant $a$ in the exponent guarantees that the value of $r$ ranges in the interval $ (0,2]$. 
An example of a periodic function represented in this visualization is shown in Figure~\ref{fig:periodic_function}.
Another visualization of a set of functions is presented in Figure~\ref{per} and the $B$-splines chosen to decompose the function in Figure~\ref{fig:periodic_function}~{\it (Top)} are shown in the same figure {\it(Bottom-left)}.
 
 A splinet is best visualized as a net of functions rather than a sequence of them, see Figure~\ref{per_net}.
 A periodic splinet is thus presented on concentric circles that correspond to the levels visualized for the regular splinets. 
 This method represents the periodic splinets for different levels on nested circles having the same center and different radii.
 If we assume that $N = \lceil\log (n/k) / \log 2\rceil$, where $k$ is the order of splines and $n$ is the number of knots, then there will be $N$-levels in the dyadic pyramid on which the splinet is presented.
 The lowest level in the splinet's pyramid contains the largest number of splines and thus is presented on a circle with the largest radius equal to $1+ 2N$. 
 Higher levels in splinets are presented on circles with radii decreased by $2$ from the  level above.
 Finally, the last, the $N$th layer contains only two $k$-tuplets of splines.
 Hence, we define a one-to-one map $T_{l}$ for each level $l=1,\dots,N$ in the net  to an rectangular region
as follows
$$
(x,y) \longrightarrow (\theta=  2 \pi x, r= 2(N-l) + \exp (a y)), 
$$
where $a$ is as before.
Although this visualization is designed to represent the dyadic orthogonal splines that we call the splinets other bases that are using the $B$-splines can be visualized this way as well, see Figure~\ref{fig:periodic_function}~{\it(bottom-right)}
for the visualization of the $B$-splines themselves. 
In this case, $N=3$ and the chosen $B$-splines are residing on three concentric levels.
We have the relation $n=k 2^N$, which is referred to as a fully dyadic case, where $n$ is the number of knots and also the dimension (the number of elements) of the basis.

 Such visualization is particularly convenient to visualize a projection to a linear space of splines.
In Figure~\ref{periodic_splinet} {\it(top-left)}, we show the splinet obtained by the dyadic orthogonalization of the $B$-splines. 
The $12$ elements of the splinet at the lowest level are visualized on the largest circle with the radius $2N-1=5$. 
The higher level is visualized on the second smaller circle with the radius $3$,  and the $6$ splines on the highest level are visualized on the smallest circle with the radius $1$. 

We use the same graphical scheme to present the splinet spectral decomposition of the periodic function shown in Figure~\ref{fig:periodic_function}~{\it (top)}. The dyadic polar coordinate structure is used for presenting the splinet. 
The corresponding coefficients of the projection of the function to the splinet are shown in 
Figure~\ref{periodic_splinet}~{\it (top-right)} shows the splinet's spectral decomposition of the function presented in Figure~\ref{fig:periodic_function}.
For comparison, an analogous spectral decomposition of the original signal in terms two-sided orthornomal splines is presented in the bottom graphs. 
Here, the dyadic structure is not used.
\section{Efficiency of the periodic splinets}
\label{eff}
The periodic splinets inherit the qualities that the splinets have, namely: locality and computational efficiency. The locality is  expressed via the small size of the total support of a periodic splinet (the sum of the sizes of the individual spline supports), that is presented in the following remark.
\begin{remark}
\label{totsup}
The periodic splinet of order $k$ defined over a dyadic set of knots $\boldsymbol \xi=(\xi_0,\dots,\xi_{n})$ (so that the dimension of the space is equal to $n$), where $n=k 2^N$ for $N\ge 0$ and $\xi_0=\xi_n$, has the relative size of the total support independent of the location of knots and equal to 
$$
k \frac{\log (2n/k)}{\log 2}.
$$ 
\end{remark}
This remark follows directly from the fact  the total support of the regular splinets equals to, see Proposition 5 in \cite{splinets}, 
$$
k \frac{\log (n/k)}{\log 2},
$$ 
in addition to the total support of the extra splinets, after orthogonalizing them with respect to the lower levels,  (the top level as in Figure \ref{per_net}) which equals to $k$.

The locality of periodic splines leads to a lower number of the inner product evaluations involving each orthogonalization. This grants computational efficiency.  The following proposition  gives a comparison between the number of  inner products that are needed to orthogonalize periodic splines in the traditional way, Gram-Schmidt,  and in our method. 
\begin{proposition}
Consider the dyadic structure case for the periodic splines of order $k$. 
Then the one-sided orthogonalization requires evaluation of 
$$
J^1_n=2nk -3 k^2 -k
$$ 
inner products, while the corresponding number for the splinet is
$$
J^2_n=\frac{5k-1}{4} n -5k^2/2 -(3k-1)/4.
$$
\end{proposition}

\begin{proof}
The proof is similar to and depending on that of Proposition 6, \cite{splinets}.
We consider the dyadic case for which $n=k2^{N}$ and there are  $n$ of periodic splines to be orthogonalized. 
The 'regular`  $n-k$ splines need $nk-3 k^2/2 -k/2$ inner products for one-sided orthogonalization. 
The extra $k$-tuplet splines, that are added to complete the periodic splinet, have common supports with all the  $n-k$ splinets. 
Hence, to orthogonalize the $k$ extra splines, we need $\sum _{j=0}^{k-1} (n-k+j)= nk -3k^2/2 -k/2$ inner products. Summing up gives the total number of the inner products $J^1_n$.

For a  periodic splinet, orthogonalization of the regular splines, without the extra $k$-tuplet splines, to get the splinet needs 
$$
\frac{5k-1}4 (n-1) -\frac{2k^2}{\log 2}\log n + \frac 4 k-3k^2+2k^2\frac{\log k}{\log 2}
$$
inner products, see Proposition 6, \cite{splinets}. Orthogonalization of the extra $k$-tuplet with respect to each other requires $k(k-1)/2$ inner products. Each spline from the extra $k$-tuplet has a common support only with the first $k$-tuplet and the last $k$-tuplet from each level. Hence, the number of inner products that is required for each level is $2k^2$ and there are $N $ levels. Consequently, the total number of the inner products that is required for a periodic splinet is $J^2_n$, where the stated rate follows from using the relations $N=\log (n/k)/\log2$.
\end{proof}

\begin{figure}[t!]
  \centering
\includegraphics[width=0.8\textwidth]{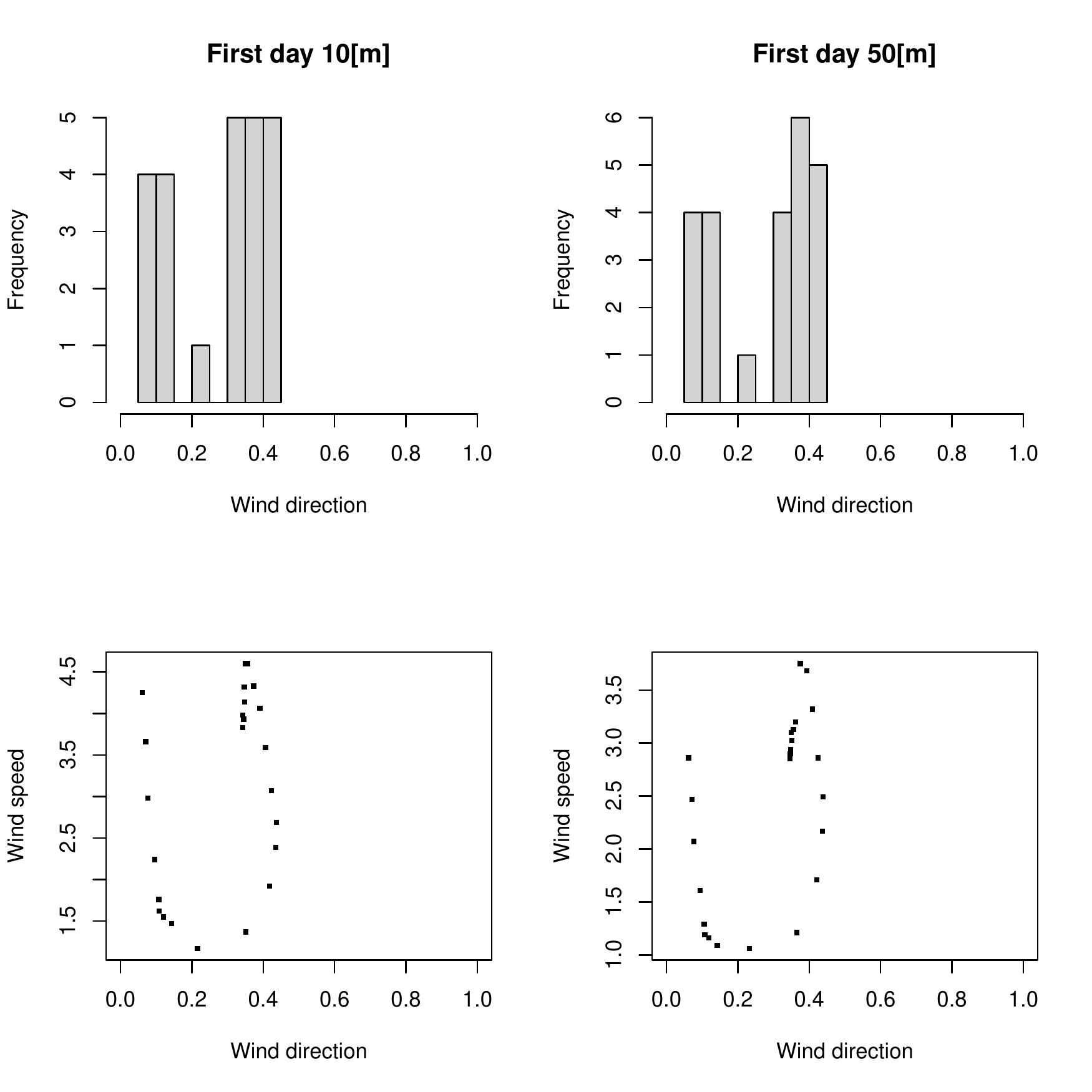}
  \caption{\small Row histrogram/scatter plot data of the first day.}
  \label{fig:raw1}
  \end{figure}

 \section{Application -- Wind dataset}
 \label{wind}
Analyzing wind direction data is important in environmental sciences, where it is used, for example, for predicting weather patterns and global climate. 
These data are circular in nature and usually represented as angles from $0^{\circ}$ to $360^{\circ}$. 
An example illustrating the convenience of our circular representation of data is made of four-variate time series that was obtained from the National Aeronautics and Space Administration (NASA) Langley Research Center (LaRC) Prediction of Worldwide Energy Resource (POWER) Project funded through the NASA Earth Science/Applied Science Program, \url{https://power.larc.nasa.gov/data-access-viewer/}. 
The four variables observed in time are: the wind directions and wind speeds measured at 10[m] and 50[m] at one point in Florida in USA from dates  01/01/2015 through 03/05/2015 at the frequency 1 per hour. The data is kept in an eight-dimensional data frame with the first quadruple corresponding to the year, month, day, and hour and the second quadruple corresponding to the wind direction and the wind speed at $10[m]$ and $50[m]$.

\begin{figure}[t!]
  \centering
\includegraphics[width=0.4\textwidth]{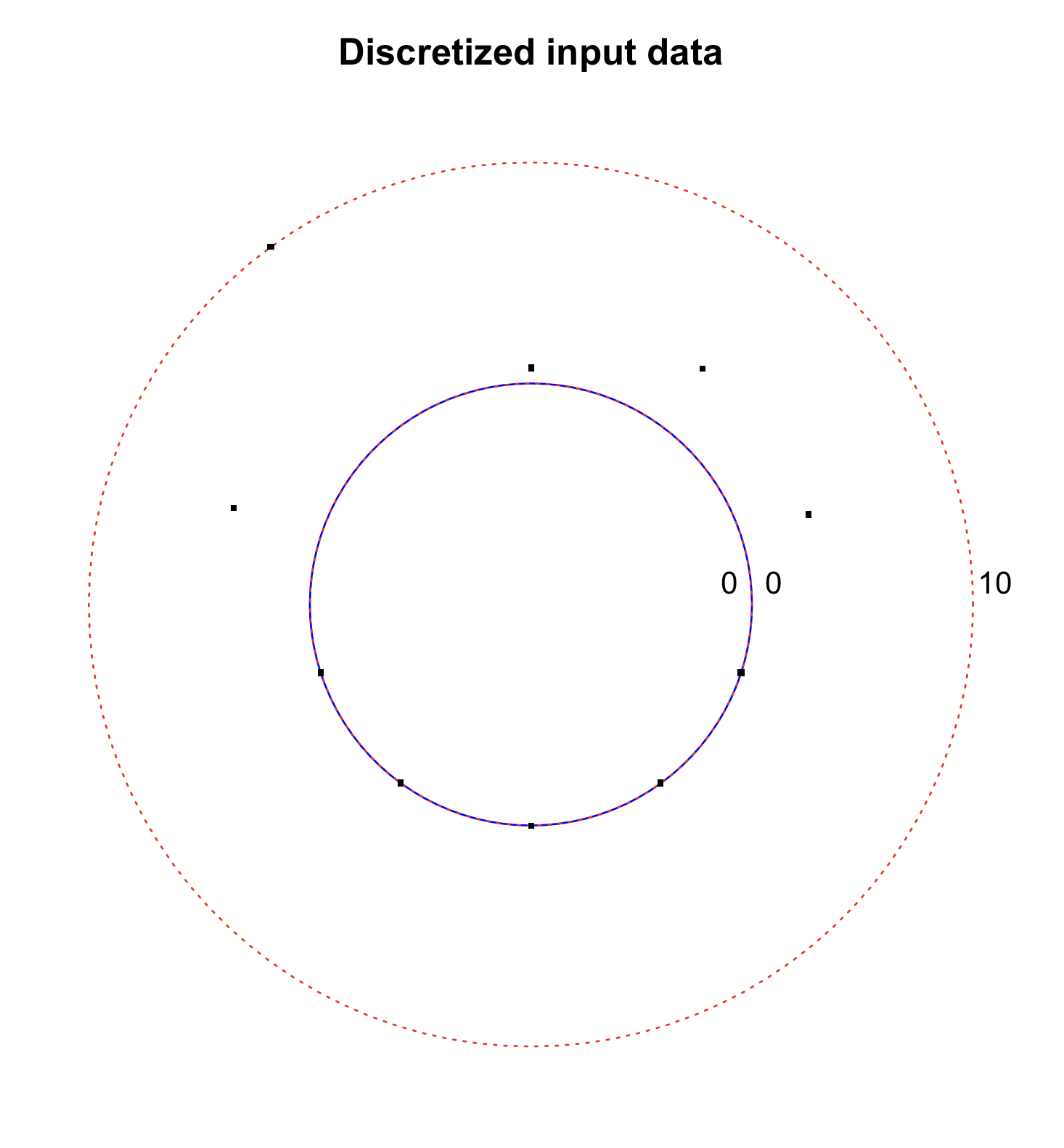}
\includegraphics[width=0.4\textwidth]{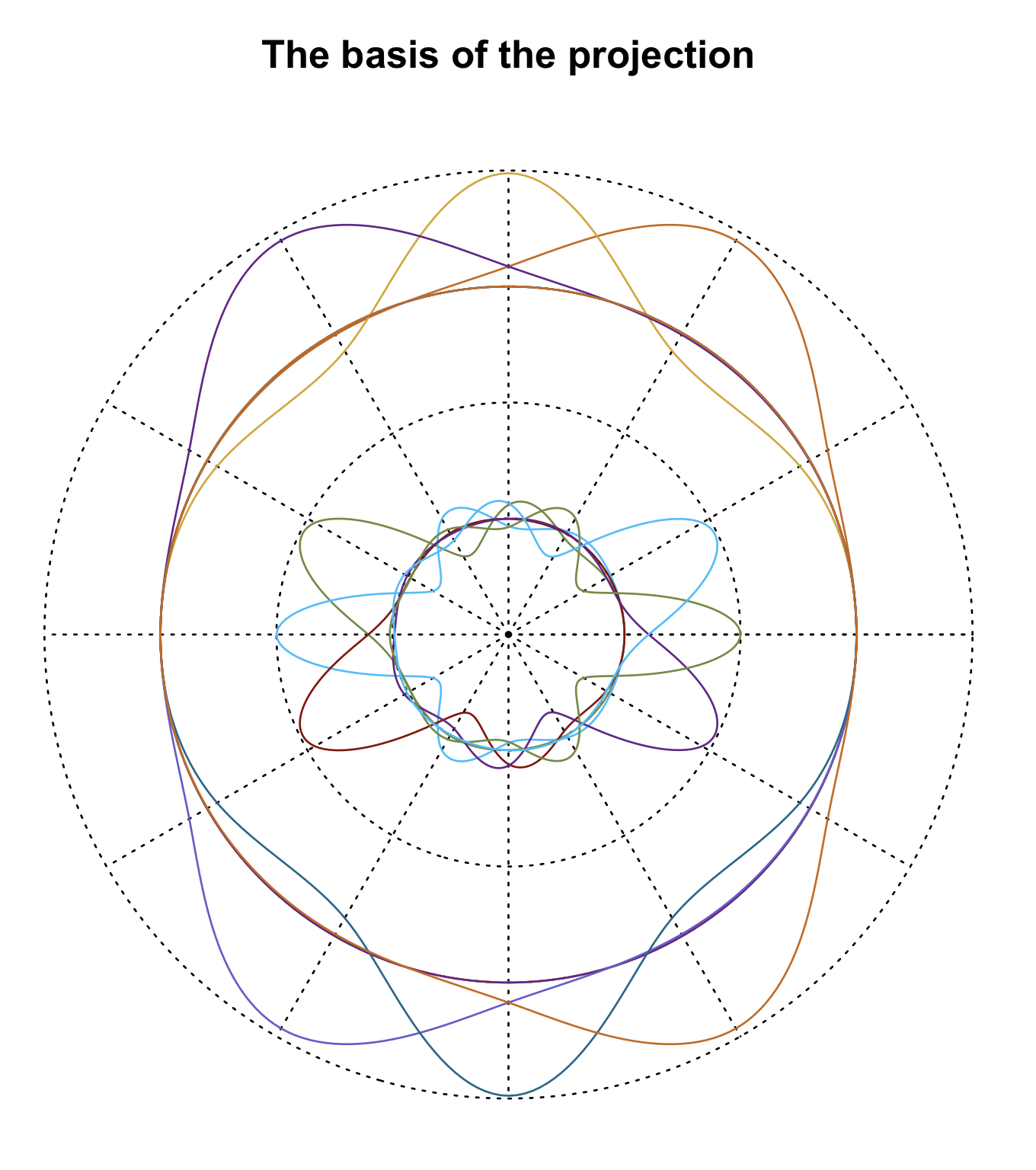}\\
\includegraphics[width=0.4\textwidth]{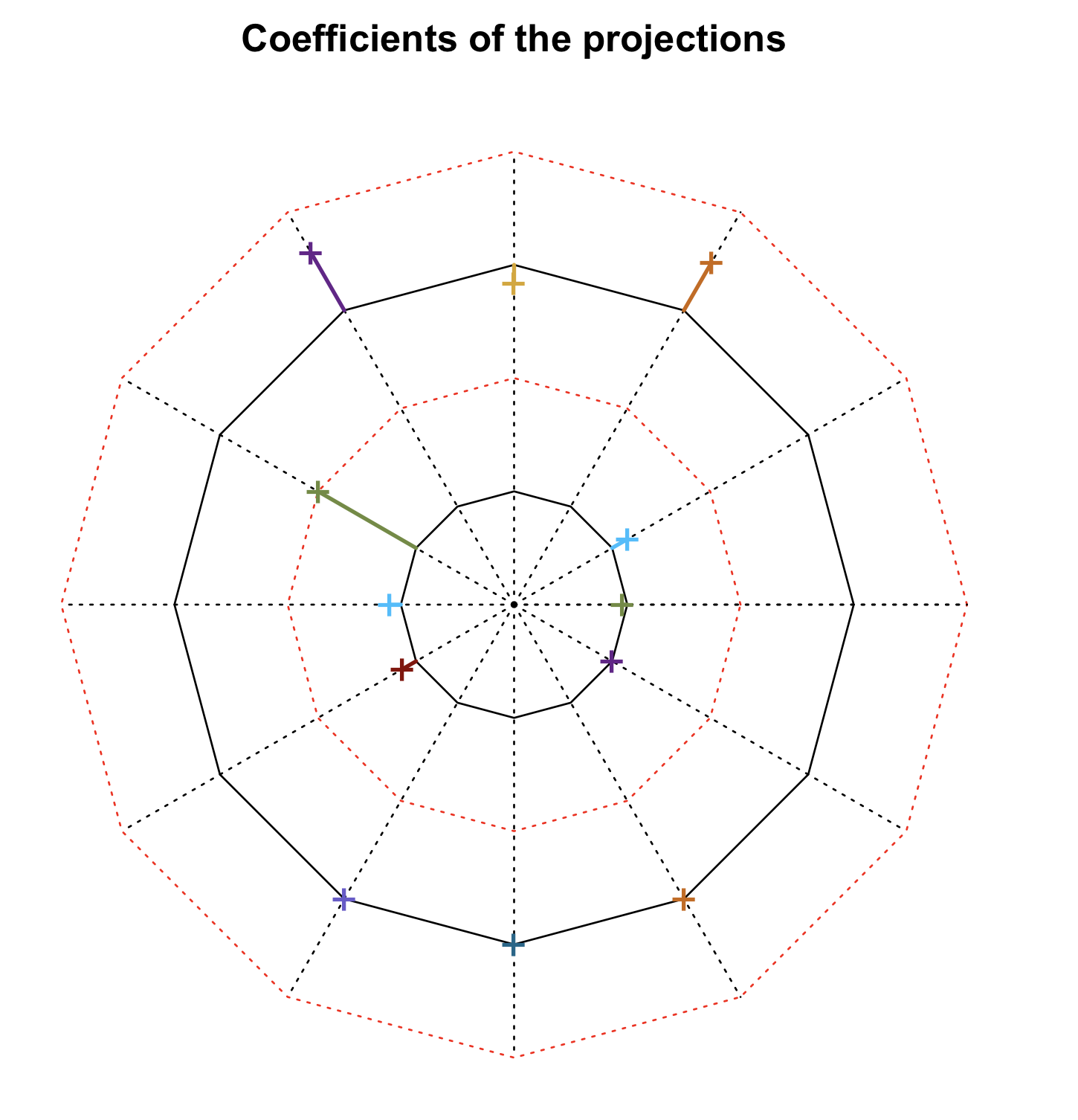}
\includegraphics[width=0.4\textwidth]{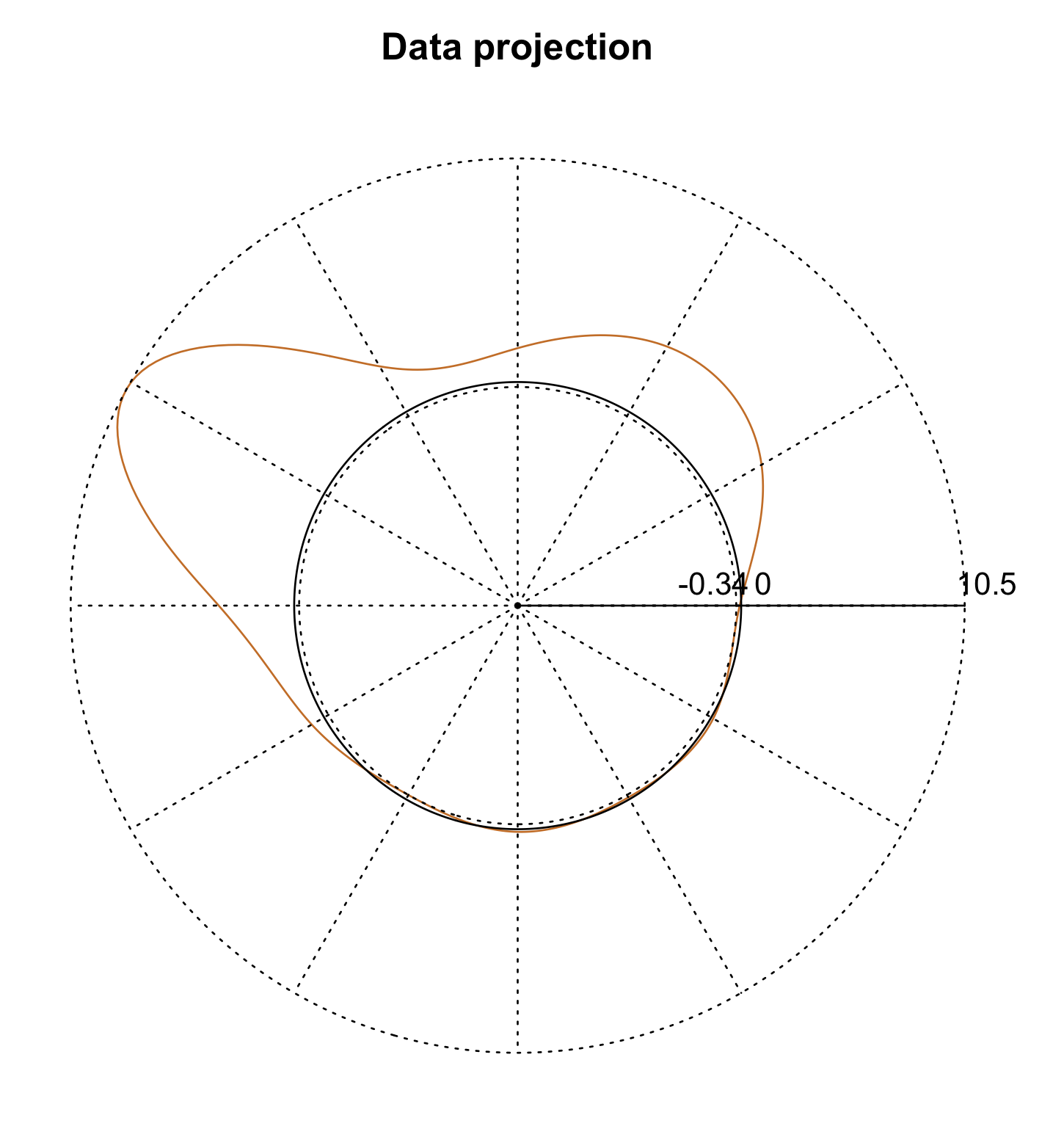}
  \caption{\small Transforming the first day 10[m] histogram to a function on a circle. {\it Left-Top:} Row histogram data of the {\it left-top} histogram in Figure~\ref{fig:raw1}; {\it Right-Top:} The splinet, of degree 3, used for the projection together with the knots marked by dashed lines; {\it Left-Bottom:} The coefficient of  the projection to the spline space; {\it Right-Bottom:}  The projection of the histogram $\hat f_1^{(10)}(\theta)$. }
  \label{fig:tran_hist}
  \end{figure}
  
One way to treat these data is to consider a four-dimensional periodic functional time series sampled daily:   $\mathbf X_i(\theta)=\left(f^{(10)}_i(\theta),v^{(10)}_i(\theta),f^{(50)}_i(\theta),v^{(50)}_i(\theta)\right)$, $\theta\in (0,360)$ where $f^{(10)}_i$, $f^{(50)}_i$ is the distribution density of the wind direction on the $i$th day at $10[m]$, $50[m]$, respectively and  $v^{(10)}_i(\theta)$, $v^{(50)}_i(\theta)$ is the value of the velocity on the $i$th day at $10[m]$, $50[m]$, respectively, given that the direction $\theta$ is observed. 
The considered data set covers 64 days, i.e.  $i=1,\dots 64$. 

Assuming this model, the daily data can be used to fit $\mathbf X_i(\theta)$ by taking an estimate of the densities $f^{(10)}$ and $f^{(50)}$ by spline smoothing of the daily histograms the directions, while the functions $v^{(10)}$ and $v^{(50)}$ can be obtained by spline fitting daily speed vs. direction. 
The raw data of the first day are seen in Figure~\ref{fig:raw1}.

The data-based functional estimates $\hat{\mathbf X}_i(\theta)=\left(\hat f^{(10)}_i(\theta),\hat v^{(10)}_i(\theta),\hat f^{(50)}_i(\theta),\hat v^{(50)}_i(\theta)\right)$ of $\mathbf X_i(\theta)$ can be obtained as follows.

 \begin{figure}[t!]
  \centering
\includegraphics[width=0.4\textwidth]{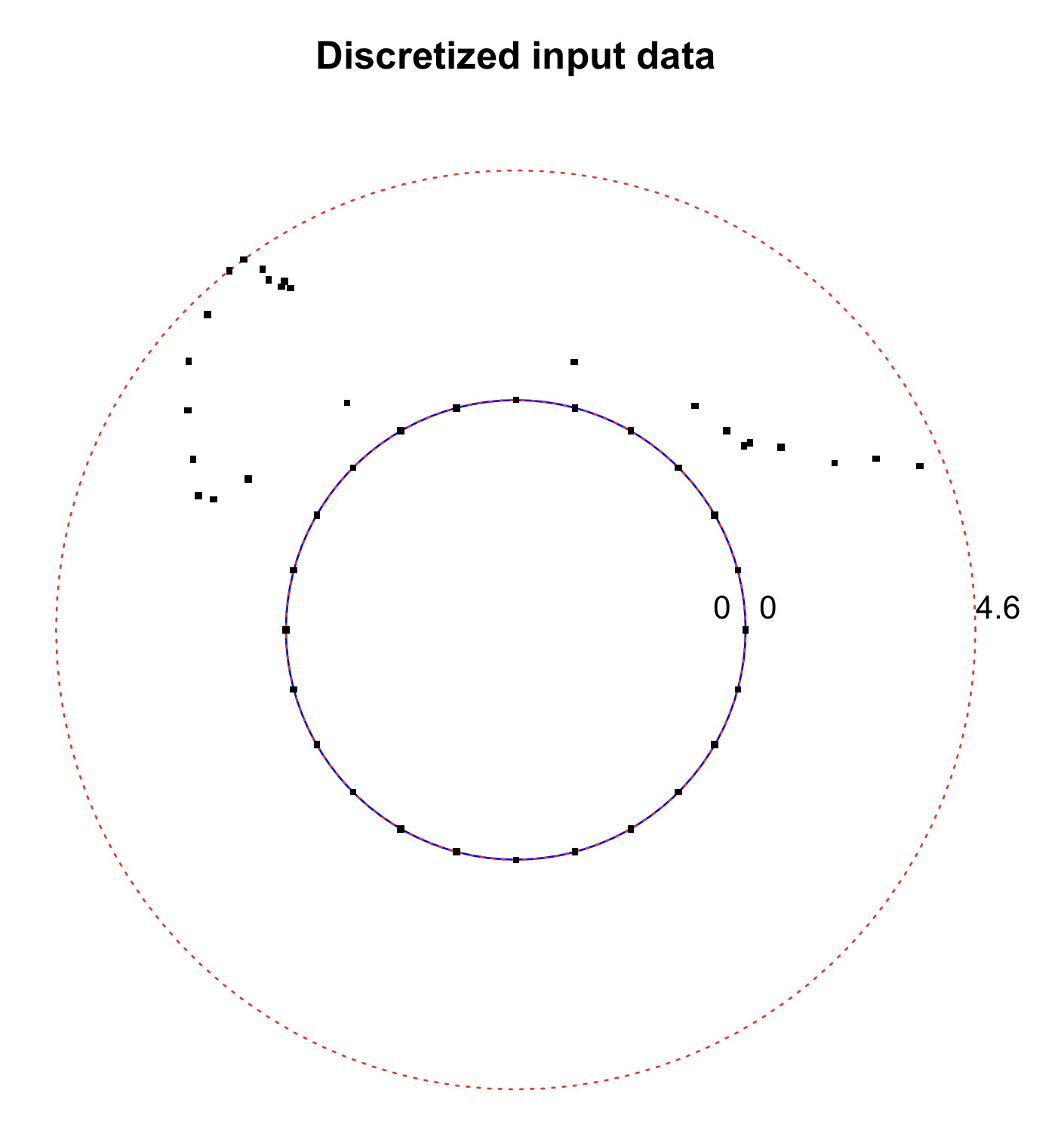}
\includegraphics[width=0.4\textwidth]{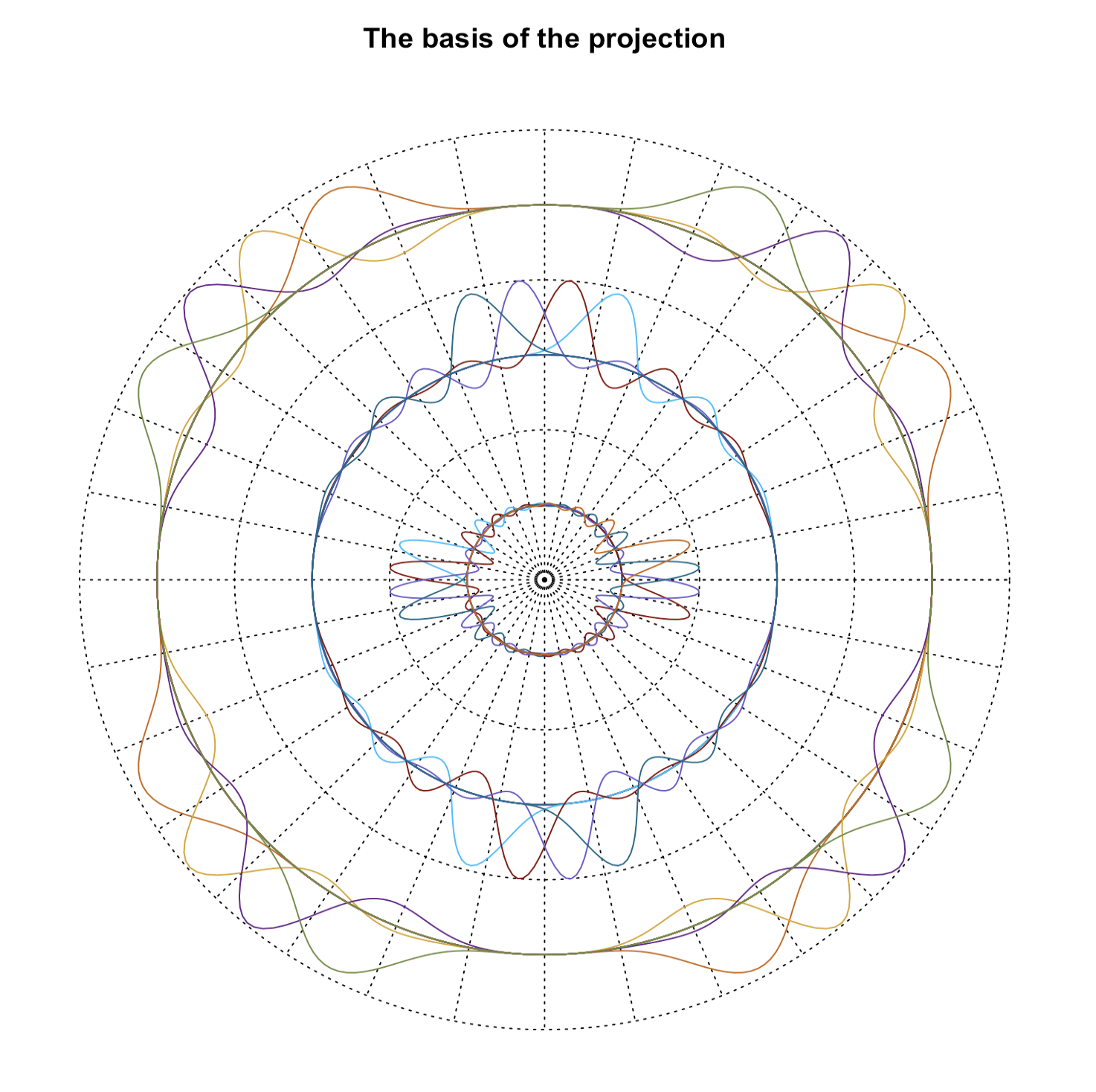}\\
\includegraphics[width=0.4\textwidth]{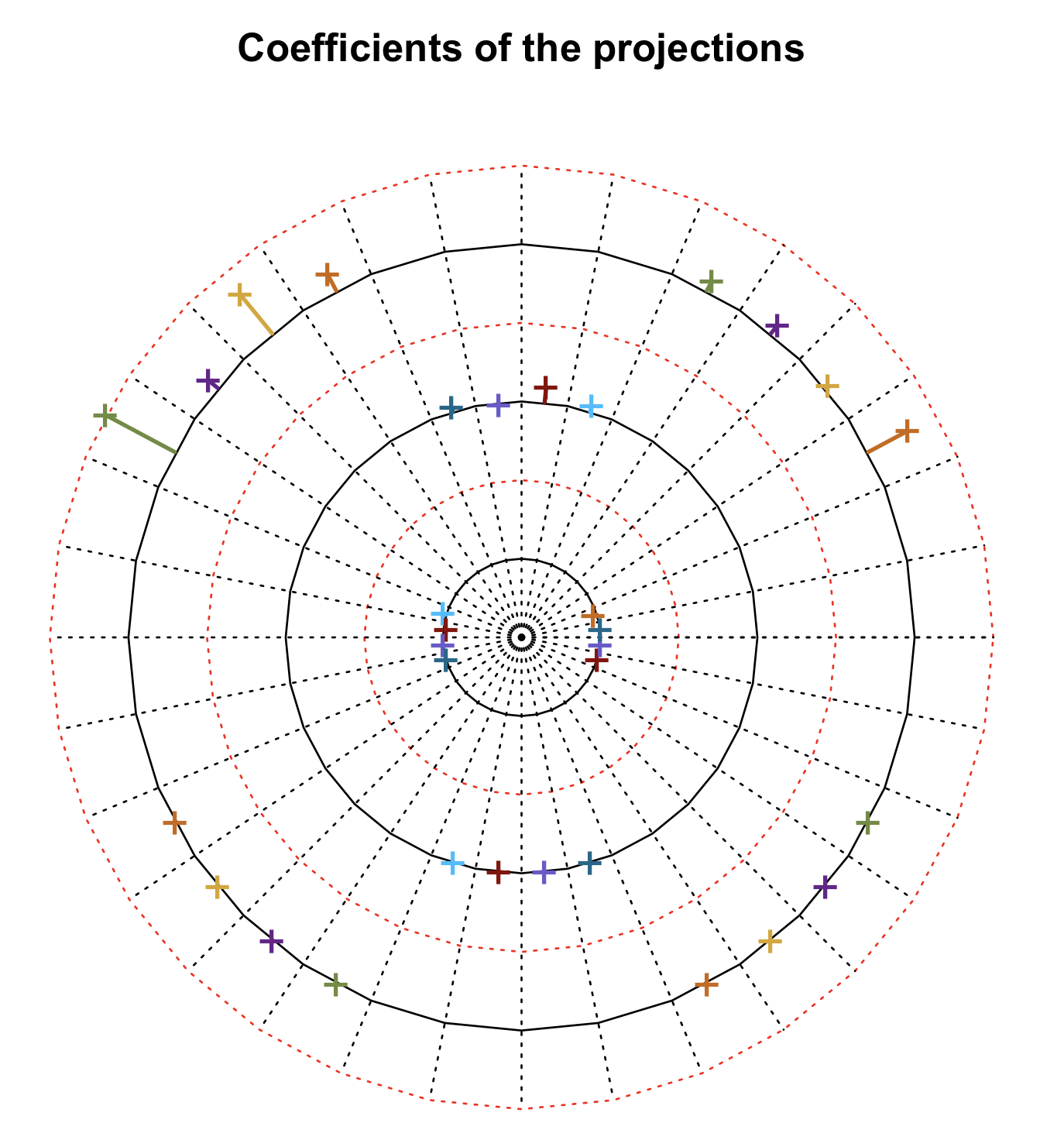}
\includegraphics[width=0.4\textwidth]{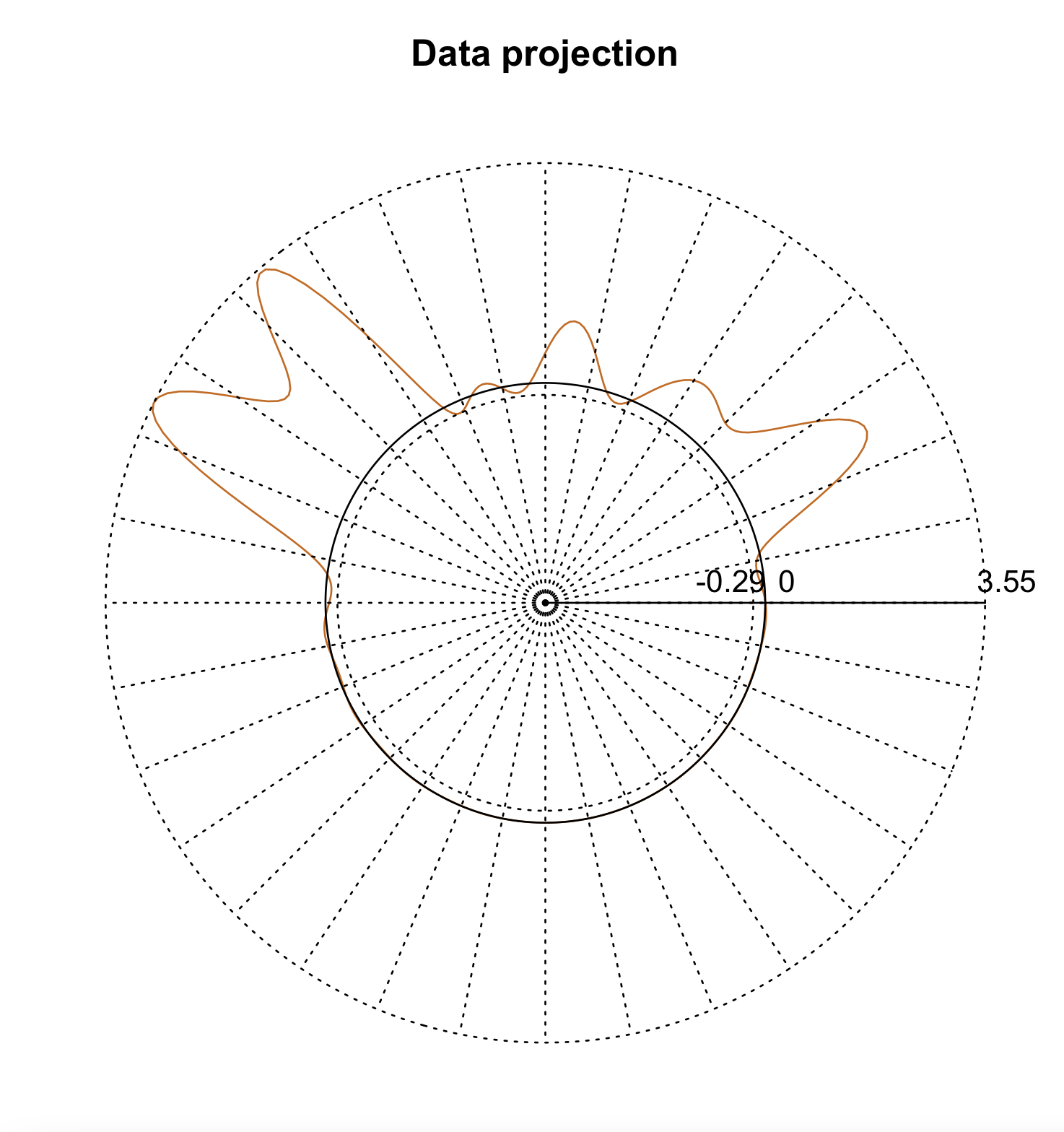}
  \caption{\small Transforming the first day 10[m] wind speed data to a function on a circle. {\it Left-Top:} Row scatter plot data in the polar coordinate representation of the {\it left-bottom} data in Figure~\ref{fig:raw1}; {\it Right-Top:} The splinet, of degree 4, used for the projection together with the knots marked by dashed lines; {\it Left-Bottom:} The coefficients of  the projection to the spline space; {\it Right-Bottom:}  The projection of the daily wind speed data $\hat v_1^{(10)}(\theta)$. }
  \label{fig:tran_speed}
  \end{figure}

Using a projection based on a splinet (function {\tt project()} in the {\tt Splinet} R-package) the daily data are transformed into continuous functions. In Figure~\ref{fig:tran_hist}, we see the components used to project the raw histogram data ({\it Left-Top}) to the smooth function ({\it Right-Bottom}).
The orthogonal spline basis, the splinet, that is used for the projection is presented in the {\it Right-Top} graph, while the corresponding coefficients of the projection are seen in the {\it Left-Bottom} figure. 
The splines are of the third order which can be seen from the graphs as the basis splines are grouped in triplets. 
We have considered the fully dyadic case corresponding to a 12-dimensional space of the third-order periodical splines that are presented the dyadic structure spanned over two levels.

To represent the wind speed in terms of the wind direction we project the scattered daily dataset to periodic splines using the same function {\tt project()}, for more details about this function see \cite{Podgorski}. The function returns a list  made of the four components:

\begin{description}
\item[onsp\$input]  – the original input data, which in the discrete data case are ordered with respect to the argument, if the original data were not, the spline input remains unchanged,
  \item[onsp\$coeff]  – the matrix of coefficients of the decomposition in the selected basis,
  \item[onsp\$basis]  – the Splinets-object representing the selected basis,
  \item[onsp\$sp]  – the Splinets-object representing the projection of the input in the projection spline space.
\end{description}

The raw data seen in Figure~\ref{fig:raw1}~{\it (Bottom)} need to be padded with zeros as the coverage of the interval $[0,1]$ is sparse. 
However, overall there are more data points than in histograms so that higher dimensional spline space is used.

Figure~\ref{fig:tran_speed} shows the results from the function {\tt project()}. 
The left-top figure presents scatter-plot of the first-day data at 10[m] in the polar coordinate format as described earlier. 
This is the same data as in Figure~\ref{fig:raw1} {(left-bottom)}. 
The remaining plots are analogous to the ones seen in Figure~\ref{fig:tran_hist}, except this time we use the dyadic structure with $N=3$ (three levels) and the fourth-order smoother splines as seen both in the representation of the splinet and its coefficients. 

The presented transformation of the data constitutes just an example of utilizing the periodic splines to prepare functional data. The actual analysis of the four-dimensional functional time series is not performed however the methods of the functional data as presented, for example, in \cite{horvatk} can be readily applied together with the tools implemented in the {\tt Splinets}-package.
Here, for the sake of promotion of the package, simple descriptive statistics aspects of the proposed functional representation of the data are obtained from the data by the means of the package.

\begin{figure}[t!]
  \centering
\includegraphics[width=0.4\textwidth]{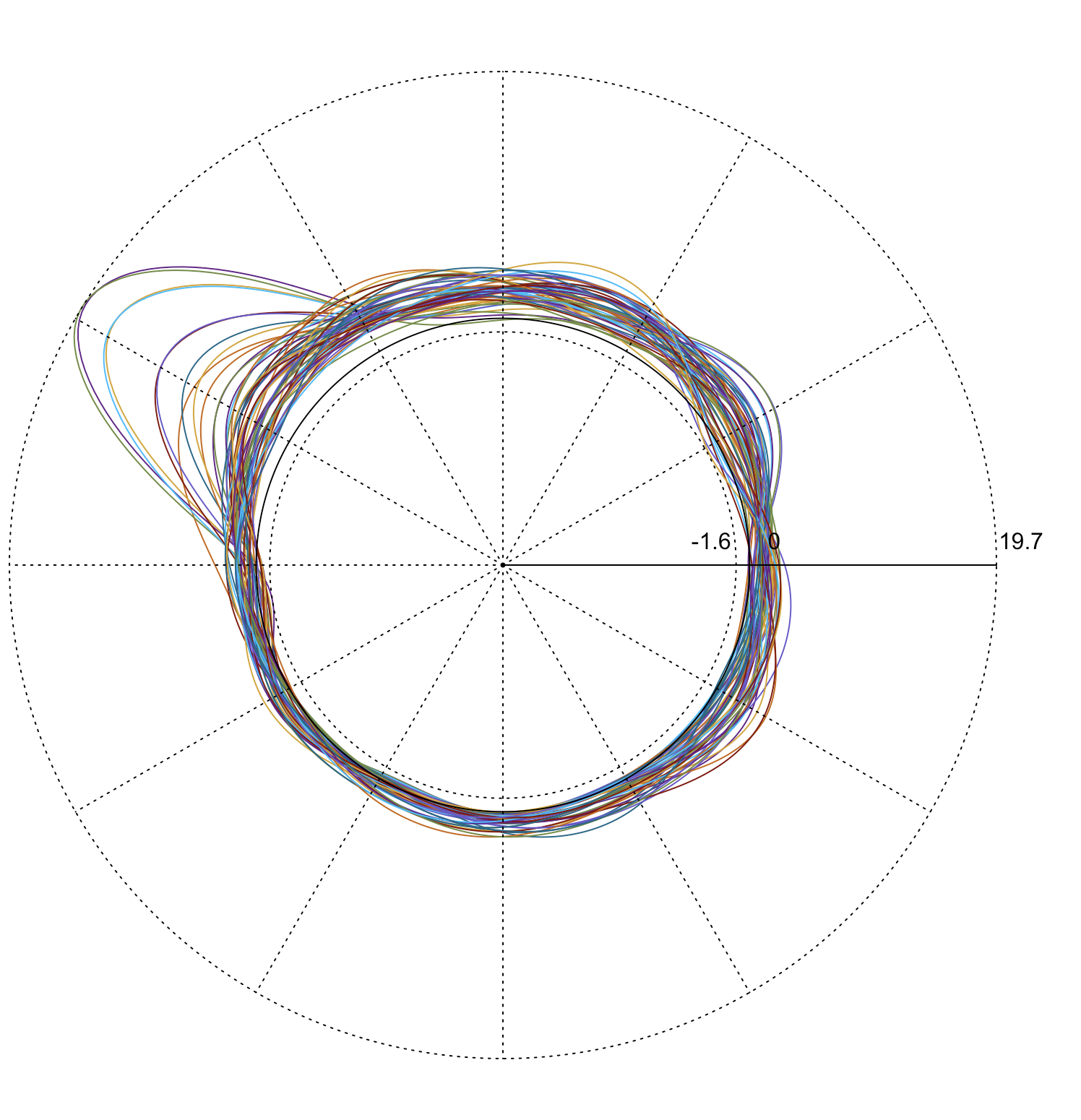}
\includegraphics[width=0.4\textwidth]{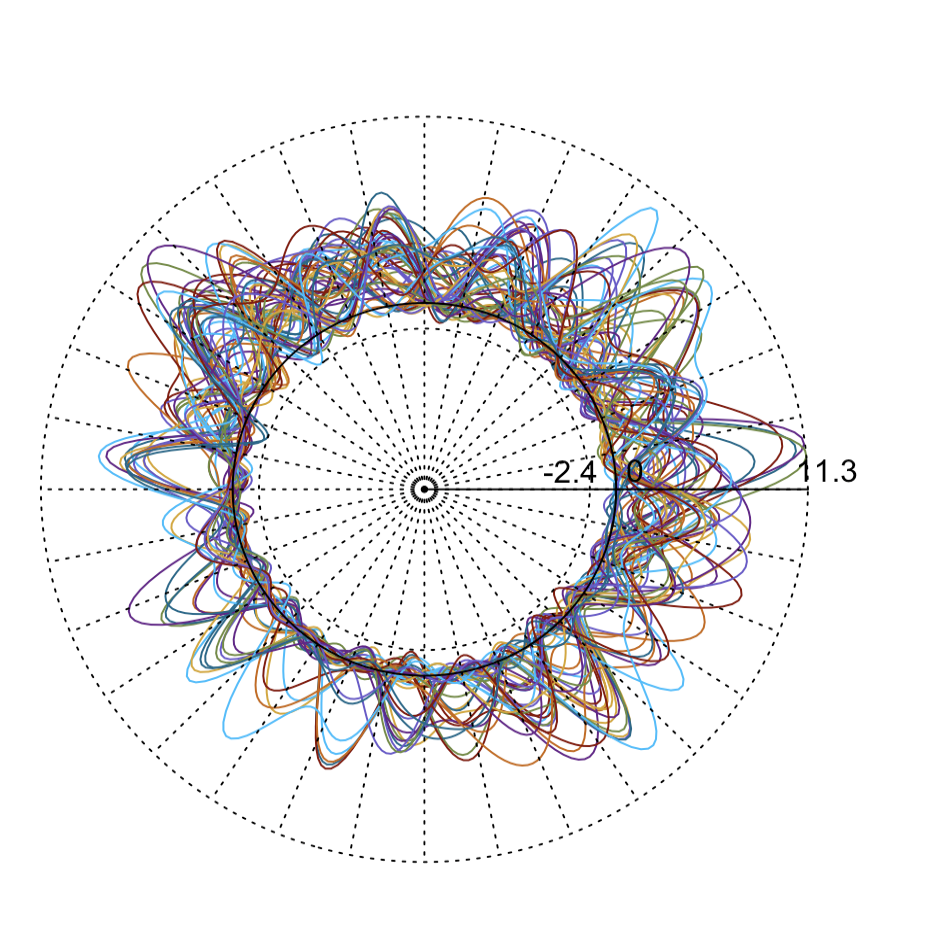}\\
\includegraphics[width=0.4\textwidth]{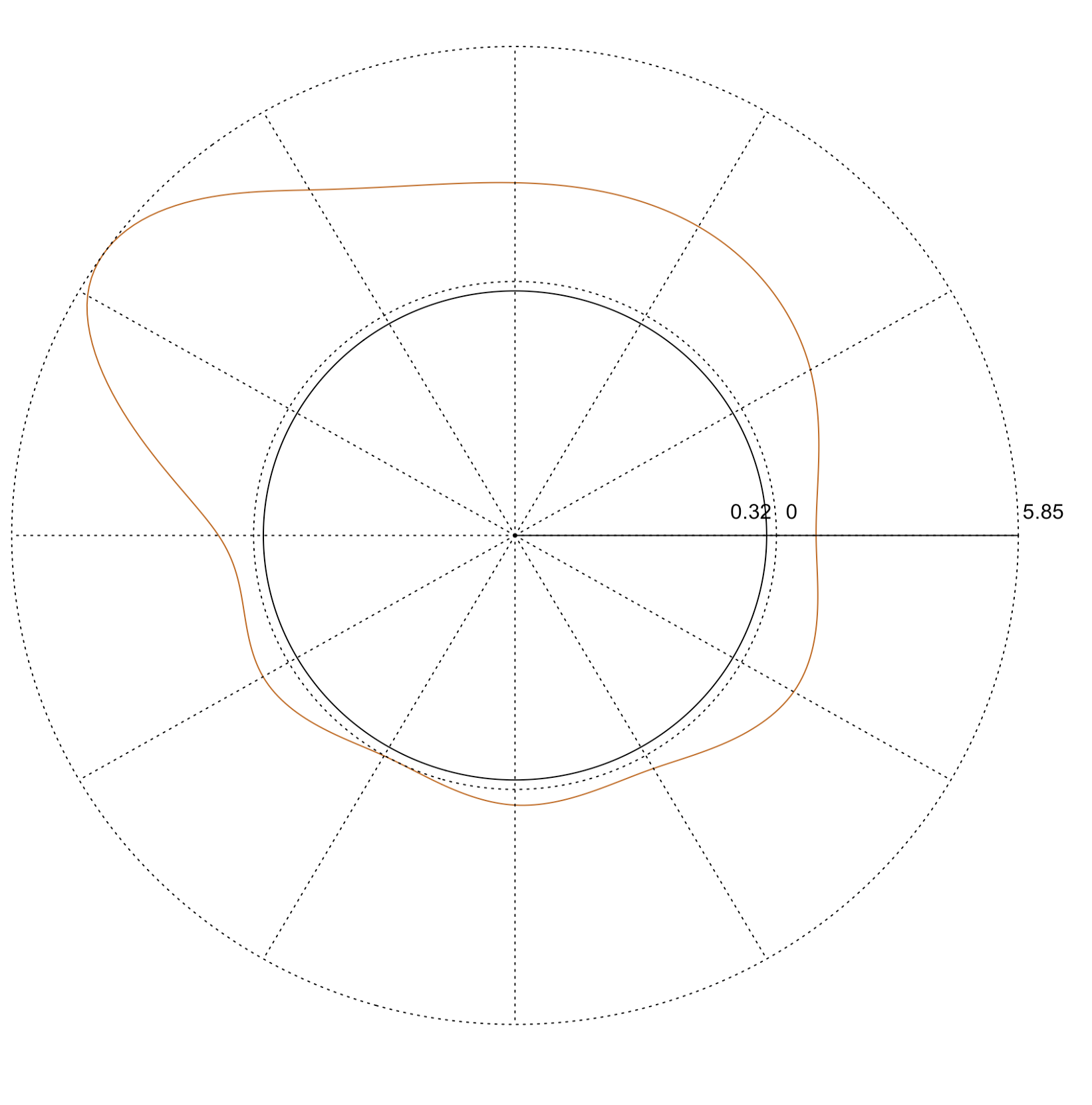}
\includegraphics[width=0.4\textwidth]{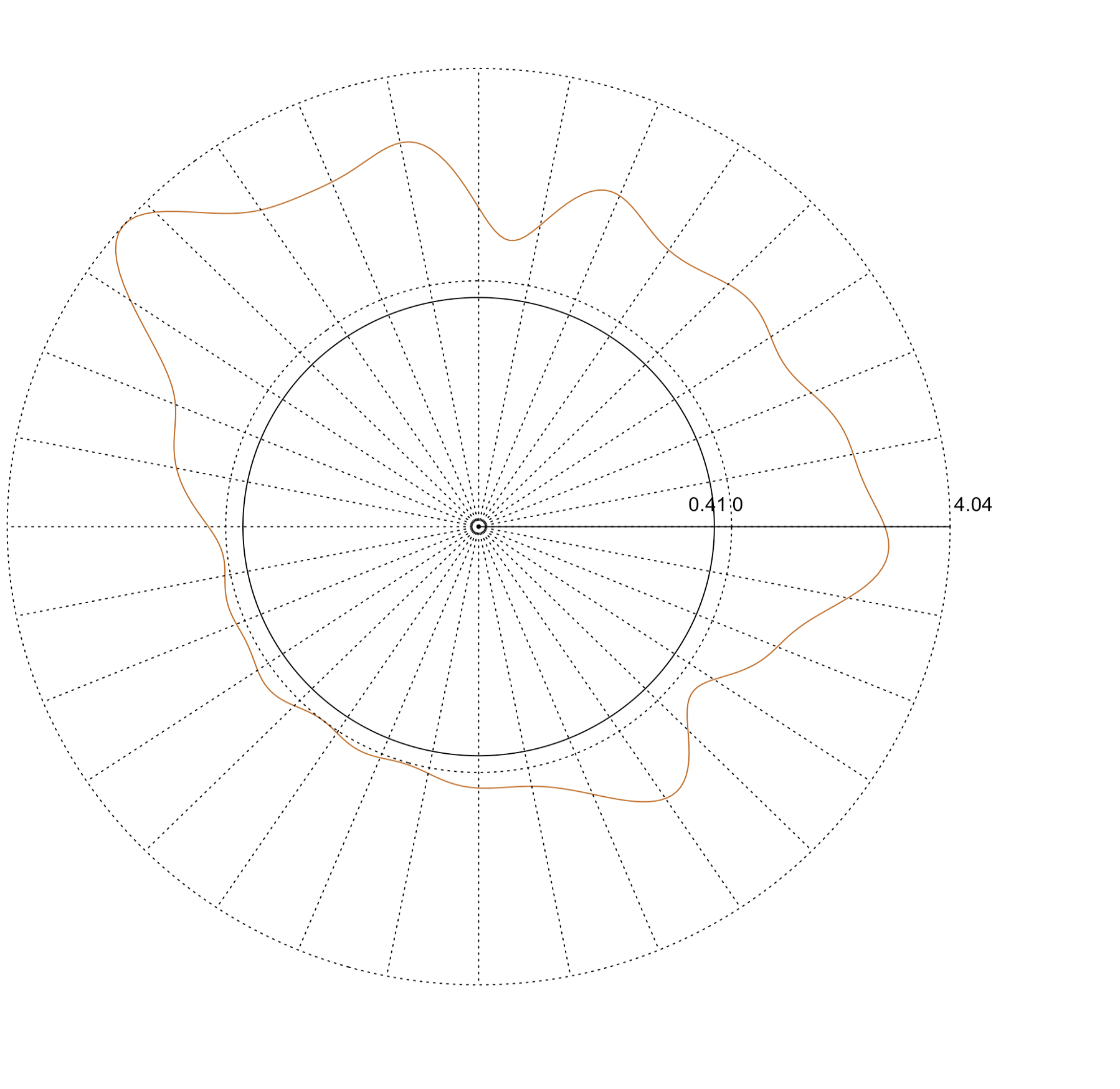}
  \caption{\small First look at the functional data corresponding to $10[m]$-height. {\it Left-Top:} 64 daily wind direction distributions; {\it Right-Top:} 64 daily dependence of the wind speed on the wind direction; {\it Left-Bottom:} The average of the daily wind-daily direction distributions; {\it Right-Bottom:}  The average wind-speed vs. the wind-direction. }
  \label{fig:all_days}
  \end{figure}
  
Only the functional daily data at $10[m]$-height are considered. The bivariate functional time series  $\left(\hat f^{(10)}_i(\theta),\hat v^{(10)}_i(\theta)\right)$, $i=1,\dots, 64$ are obtained from the data as described above. In the top two graphs, the functional data are plotted on the common graph, the left-hand-side graph corresponding to the distribution of the wind direction, and the right-hand-side graph showing the daily dependences of wind speed on a direction. 
One can observe that the dominant frequency of wind direction also has relatively strong winds. However, there are also strong winds at the directions around zero azimuth for which frequencies are moderate. 
This feature is even better illustrated on the bottom two graphs, where the means of the functional data are presented. In particular, for the azimuth between $180^\circ$ and $270^\circ$, we observe neither frequent winds nor strong ones.
Any further analysis of the data is beyond the purpose of this presentation.

\bibliography{references.bib}
